\RequirePackage[l2tabu,orthodox]{nag}		
\PassOptionsToPackage{numbers,sort&compress}{natbib}		
\documentclass{article}
\usepackage[final]{neurips_2020}

\usepackage{amsmath}		
\usepackage{amssymb}		
\usepackage{amsfonts}		
\usepackage{amsthm}		    

\usepackage{mathtools}		

\usepackage[utf8]{inputenc}		
\usepackage[T1]{fontenc}		
\renewcommand\bf{\bfseries}

\usepackage{dsfont}		

\usepackage{inconsolata}

\usepackage[
cal=cm,
]
{mathalfa}

\usepackage[labelfont={bf,small},labelsep=colon,font=small, singlelinecheck=false]{caption}	

\usepackage[dvipsnames,svgnames]{xcolor}		
\colorlet{MyBlue}{DodgerBlue!75!Black}
\colorlet{MyGreen}{DarkGreen!85!Black}

\usepackage{titlesec}
\usepackage{placeins}

\usepackage{wasysym}		

\usepackage{wrapfig}      
\usepackage{subcaption}		
\usepackage{tikz}		
\usetikzlibrary{calc,patterns}

\usepackage{array}		
\newcolumntype{C}[1]{>{\centering\let\newline\\\arraybackslash\hspace{0pt}}m{#1}}
\usepackage{booktabs}		
\usepackage[inline,shortlabels]{enumitem}		
\usepackage{multirow}
\usepackage{threeparttable}   
\usepackage{tabularx}

\setlist[enumerate]{leftmargin=*}
\setlist[itemize]{leftmargin=*}   

\usepackage{microtype}		

\usepackage{acronym}		
\usepackage{latexsym}		
\usepackage{xfrac}		
\usepackage{xspace}		
\usepackage{xparse}     

\usepackage{hyperref}
\hypersetup{
colorlinks=true,
linktocpage=true,
pdfstartview=FitH,
breaklinks=true,
pdfpagemode=UseNone,
pageanchor=true,
pdfpagemode=UseOutlines,
plainpages=false,
bookmarksnumbered,
bookmarksopen=false,
bookmarksopenlevel=1,
hypertexnames=true,
pdfhighlight=/O,
urlcolor=Maroon,linkcolor=MyBlue!60!black,citecolor=MyBlue!70!black,	
pdftitle={},
pdfauthor={},
pdfsubject={},
pdfkeywords={},
pdfcreator={pdfLaTeX},
pdfproducer={LaTeX with hyperref}
}
\usepackage[sort&compress,capitalize,nameinlink]{cleveref}		
\crefname{assumption}{Assumption}{Assumptions}
\crefname{assumptionloc}{Assumption}{Assumptions}

\usepackage[textwidth=20mm]{todonotes}		
\newcommand{\para}[1]{\paragraph{#1\afterhead}}

\newcommand{\debug}[1]{#1}		

\usepackage{thmtools,thm-restate,thm-autoref}

\theoremstyle{plain}
\newtheorem{theorem}{Theorem}		
\newtheorem{lemma}{Lemma}		
\newtheorem{proposition}{Proposition}		

\newtheorem*{corollary*}{Corollary}		

\theoremstyle{definition}
\newtheorem{assumption}{Assumption}		

\newtheorem*{definition*}{Definition}		
\newtheorem*{assumption*}{Assumptions}		

\makeatletter		
\newcommand{\asmtag}[1]{
  \let\oldtheassumption\theassumption
  \renewcommand{\theassumption}{#1}
  \g@addto@macro\endassumption{
    \addtocounter{assumption}{0}
    \global\let\theassumption\oldtheassumption}
  }
\makeatother

\theoremstyle{remark}
\newtheorem*{remark*}{Remark}		

\newcounter{proofpart}

\usepackage{macros}

\begin{document}










\title{\Large Explore Aggressively, Update Conservatively:\\[0.1em]
Stochastic Extragradient Methods with Variable Stepsize Scaling}

\author{%
Yu-Guan Hsieh\\
Univ. Grenoble Alpes, LJK\\
38000 Grenoble, France.\\
\texttt{yu-guan.hsieh@univ-grenoble-alpes.fr}
\And
Franck Iutzeler\\
Univ. Grenoble Alpes, LJK\\
38000 Grenoble, France.\\
\texttt{franck.iutzeler@univ-grenoble-alpes.fr}
\And
Jérôme Malick\\
CNRS, LJK\\
38000 Grenoble, France.\\
\texttt{jerome.malick@univ-grenoble-alpes.fr}
\And
Panayotis Mertikopoulos\\
Univ. Grenoble Alpes, CNRS, Inria, Grenoble INP, LIG\\
38000 Grenoble, France.\\
Criteo AI Lab, France\\
\texttt{panayotis.mertikopoulos@imag.fr}}

\maketitle

\begin{abstract}
%
%
Owing to their stability and convergence speed, \acl{EG} methods have become a staple for solving large-scale saddle-point problems in machine learning.
The basic premise of these algorithms is the use of an extrapolation step before performing an update;
thanks to this exploration step, \acl{EG} methods overcome many of the non-convergence issues that plague gradient descent/ascent schemes.
On the other hand, as we show in this paper, running vanilla \acl{EG} with stochastic gradients may jeopardize its convergence, even in simple bilinear models.
To overcome this failure, we investigate a double stepsize \acl{EG} algorithm where the exploration step evolves at a more aggressive time-scale compared to the update step.
We show that this modification allows the method to converge even with stochastic gradients, and we derive sharp convergence rates under an error bound condition.
\end{abstract}

\newacro{LHS}{left-hand side}
\newacro{RHS}{right-hand side}
\newacro{iid}[i.i.d.]{independent and identically distributed}
\newacro{lsc}[l.s.c.]{lower semi-continuous}

\newacro{GAN}{generative adversarial network}
\newacro{SFO}{stochastic first-order oracle}

\acresetall		

\section{Introduction}
\label{sec:intro}

A major obstacle in the training of \acp{GAN} is the lack of an implementable, strongly convergent method based on stochastic gradients. 
The reason for this is that the coupling of two (or more) neural networks gives rise to behaviors and phenomena that do not occur when minimizing an \emph{individual} loss function, irrespective of the complexity of its landscape.
As a result, there has been significant interest in the literature to codify the failures of \ac{GAN} training, and to propose methods that could potentially overcome them.

Perhaps the most prominent of these failures is the appearance of cycles \citep{DISZ18,MPP18,MLZF+19,GBVV+19,FVP19} and, potentially, the transition to aperiodic orbits and chaos \citep{HS98,San10,PS14,PPP17,CP19}.
Surprisingly, non-convergent phenomena of this kind are observed even in very simple saddle-point problems such as two-dimensional, unconstrained bilinear games
\citep{DISZ18,MLZF+19,GBVV+19}.
In view of this, it is quite common to examine the convergence (or non-convergence) of a gradient training scheme in bilinear models before applying it to more complicated, non-convex/non-concave problems.

A key observation here is that the non-convergence of standard gradient descent-ascent methods in bilinear saddle-point problems can be overcome by incorporating a ``gradient extrapolation'' step before performing an update.
The resulting algorithm, due to \citet{Kor76}, is known as the \acdef{EG} method,
and it has a long history in optimization;
for an appetizer, see \citet{FP03}, \citet{JNT11}, \citet{Nem04}, \citet{Nes07}, and references therein.
In particular, the \acl{EG} algorithm converges for all pseudomonotone \aclp{VI} (a large problem class that contains all bilinear games, \cf \cite{Kor76}), and the time-average of the generated iterates achieves an $\bigoh(1/\run)$ rate of convergence in monotone problems \citep{Nem04}.

The above concerns the application of \acl{EG} methods with perfect, \emph{deterministic} gradients and a non-vanishing stepsize.
By contrast, in the type of saddle-point problems that are encountered in machine learning (\acp{GAN}, robust reinforcement learning, etc.), there are two important points to keep in mind:
First, the size of the datasets involved precludes the use of full gradients (for more than a few passes at least), so the method must be run with \emph{stochastic} gradients instead.
Second, because the landscapes encountered are not convex-concave, the method's last iterate is typically preferred to its time-average (which offers no tangible benefits when Jensen's inequality no longer applies).
We are thus led to the following questions:
\begin{enumerate*}
[(\itshape i\hspace*{.5pt}\upshape)]
\item
\emph{are the superior last-iterate convergence properties of the \ac{EG} algorithm retained in the stochastic setting?}
And, if not,
\item
\emph{is there a principled modification that would restore them?}
\end{enumerate*}

\para{Our contributions}

To motivate our analysis, we first analyse a counterexample to show that the last iterate of stochastic \ac{EG} fails to converge, even in bilinear min-max problems where deterministic \ac{EG} methods converge from any initialization.
We then consider a class of \acdef{DSEG} methods with an exploration step evolving more aggressively than the update step and prove it enjoys better convergence guarantees than standard \ac{EG} in stochastic problems. In more detail:

\begin{enumerate}
    \item We show that the \ac{DSEG} algorithm converges with probability $1$ in a large class of problems that contains all monotone saddle-point problems.
    \item We derive explicit convergence rates for the algorithm's last iterate under an error bound condition. This is the first time that such condition is considered in the analysis of stochastic \ac{EG} methods, albeit its popularity in the optimization community.
    \item For bilinear min-max problems in particular, our analysis establishes that stochastic \ac{DSEG} methods converge at a $\bigoh(1/\run)$ rate.
    Prior to our work, last-iterate convergence rate for bilinear min-max games had only been studied in the deterministic setting.\footnote{%
    Let us still mention the work of\;\citet{LBJV+20} which appeared on arxiv a few weeks after the submission of our manuscript: it proved that stochastic Hamiltonian methods applied to (sufficiently) bilinear games ensures also a $\bigoh(1/\run)$ convergence rate. Nonetheless, Hamiltoninan gradient descent is not guaranteed to converge to a solution in monotone games
    and in general when it converges, it may converge to an unstable stationary point.}
    \item To account for non-monotone problems, we also provide local versions of these results that hold with (arbitrarily) high probability.
    Importantly, thanks to the use of a local error bound condition,
    we can obtain local convergence rates even if the Jacobian at a solution contains purely imaginary eigenvalues.
\end{enumerate}

\newcommand\Tstrut{\rule{0pt}{2.6ex}}         
\newcommand\Bstrut{\rule[-0.9ex]{0pt}{0pt}}   

\begin{table}[t]
\small
\centering
\renewcommand{\arraystretch}{1.1}
\begin{tabular}{lcccc}
\toprule
    &
	& Assumption
	& Guarantee
	& Rate\\
	\midrule
	\multirow{3}{*}{\shortstack[l]{Extragradient\\(Mirror-prox)}}
	& \cite{JNT11}
	& monotone & ergodic & $1/\sqrt{\run}$ \\
	& \cite{KS19} & strongly monotone & last & $1/\run$ \\
	& \cite{MLZF+19} & strictly coherent & last & asymptotic \Bstrut\\
    \hline\Tstrut
	\multirow{2}{*}{Increasing batch size} &
	\multirow{2}{*}{\cite{IJOT17}} &
	\multirow{2}{*}{pseudo-monotone} & best & $1/\sqrt{\run}$ \Bstrut\\
	& & & last & asymptotic \\
    \hline\Tstrut
	Repeated sampling & \cite{MKSR+20} & monotone & ergodic & $1/\sqrt{\run}$ \Bstrut\\
	\hline\Tstrut
	SVRE & \cite{CGFLJ19} & strongly monotone + finite sum & last & $e^{-\rho\run}$ \Bstrut\\
	\hline\Tstrut
	\multirow{3}{*}{Double stepsize}
	& \multirow{3}{*}{Ours}
	& variational stability (VS) & last & asymptotic \\
	&& VS + error bound & last & $1/\run^{1/3}$\\
	&& monotone + affine & last & $1/\run$ \\
	\bottomrule
\end{tabular}
\vspace{2ex}
\caption{Summary of known convergence results of stochastic extragradient methods.
For ergodic, last iterate and best iterate guarantees, the 
convergence metrics are respectively dual gap, squared distance to the solution set and squared residual.
Results for single-call 
\cite{HIMM19,LMRZ+20} and non-extragradient methods \cite{KNS12,LBJV+20,RYY19} are not included.} 
\label{tab:overview}
\vskip -1.5em
\end{table}

\para{Related works}

The approaches that have been explored in the literature to ensure the convergence of stochastic first-order methods, in monotone problems and beyond, include variance reduction with increasing batch size and schemes with vanishing regularization (or ``anchoring'').
In regard to the former, \citet{IJOT17} showed that using increasing batch size can ensure convergence in pseudomonotone \aclp{VI}.
As for the latter, \citet{KNS12} and \citet{RYY19} regularized the problem via the addition of a strongly monotone term with vanishing weight;
by properly controlling the weight reduction schedule of this regularization term, it is possible to show the method's convergence in monotone problems.

In contrast to the above, our approach is based on a modification of the choice of the stepsizes, which has only been studied theoretically in the deterministic setting.
\citet{ZY20} recently examined the convergence of several gradient-based algorithms in unconstrained zero-sum bilinear games with deterministic oracle feedback.
Interestingly, they show that the optimal (geometric) rate of convergence in bilinear games is recovered for asymptotically large ``exploration'' parameters $\step\to\infty$ and infinitesimally small ``update'' parameters $\stepalt\to0$.
Even though the setting there is quite different from our own, it is interesting to note that the principle of a smaller update stepsize also applies in their case \textendash\ see also \citet{LS19} and \citet{MKSR+20} for a concurrent series of results, and \citet{RYY19} for an empirical investigation into the stochastic setting.



Regarding convergence counterexamples, in a recent paper, \citet{CGFLJ19} showed that if \ac{EG} is run with a \emph{constant} stepsize and noise with \emph{unbounded} variance, the method's iterates actually diverge at a geometric rate. 
Motivated by this, they proposed a SVRG-type variance reduced \ac{EG} method for finite-sum problems and proved a geometric convergence of the algorithm when the involved operator is strongly monotone.
Compared to this situation, our counterexample illustrates that the non-convergence persists for \emph{any} error distribution with positive variance (no matter how small) and \emph{any} stepsize sequence (constant, decreasing, or otherwise).
In particular, if \ac{EG} is run with noisy feedback, its trajectories remain non-convergent even if the noise is almost surely bounded and a vanishing stepsize schedule is employed.

Finally, to make our paper's position clear with respect to the large corpus of work on stochastic \ac{EG} methods, we further provide an overview of the most relevant results in \cref{tab:overview} and refer the interested reader to the supplement for further discussion.

\section{Preliminaries}
\label{sec:setup}

In this section, we briefly review some basics for the class of problems under consideration \textendash\ namely, saddle-point problems and the associated vector field formulation.

\para{Saddle-point problems}

The flurry of activity surrounding the training of \acp{GAN} has sparked renewed interest in saddle-point problems and zero-sum games.
To define this class of problems formally, consider a value function $\sadobj\from\minvars \times \maxvars\to\R$ which assigns a cost of $\sadobj(\minvar,\maxvar)$ to a player controlling $\minvar\in\minvars$, and a payoff of $\sadobj(\minvar,\maxvar)$ to a player choosing $\maxvar\in\maxvars$.
Then, the \emph{saddle-point problem} associated to a $\sadobj$ consists of finding a profile $(\sol[\minvar], \sol[\maxvar]) \in \minvars \times \maxvars$ such that,
for all $\minvar\in\minvars$, $\maxvar\in\maxvars$, we have: 
\begin{equation}
\tag{SP}
\label{eq:saddle}
\sadobj(\sol[\minvar], \maxvar)
	\leq \sadobj(\sol[\minvar], \sol[\maxvar])
	\leq \sadobj(\minvar, \sol[\maxvar]).
\end{equation}

In this setting, the pair $(\sol[\minvar], \sol[\maxvar])$ is called a (global) \emph{saddle point} of $\sadobj$ \textendash\ or, in game-theoretic terminology, a \acdef{NE}.
For concision and generality, we will often abstract away from $\minvar$ and $\maxvar$ by setting $\point = (\minvar,\maxvar) \in \R^{\vdim}$ (where, in obvious notation, $\vdim = \vdim_{1} + \vdim_{2}$).

\paragraph{Vector field formulation\afterhead}

In most cases of interest, the objective $\sadobj$ is differentiable and is usually accessed through a first-order oracle returning values of the vector field
$\vecfield(\minvar,\maxvar)
	= (\nabla_{\minvar}\sadobj(\minvar,\maxvar), - \nabla_{\maxvar}\sadobj(\minvar,\maxvar)).$
As usual for gradient-based methods, we will frequently (though not always) assume that $\vecfield$ is \emph{Lipschitz continuous:}

\begin{assumption}
\label{asm:Lipschitz}
The field $\vecfield$ is $\lips$-Lipschitz continuous \ie for all $\point,\pointalt\in\vecspace$,
\begin{equation}
\label{eq:Lipschitz}
\tag{LC}
\dnorm{\vecfield(\pointalt) - \vecfield(\point)}
	\leq \lips \norm{\pointalt - \point}.
\end{equation}
\end{assumption}

The importance of the above is that \eqref{eq:saddle} is often intractable, so it is natural to examine instead the first-order stationarity conditions for $\vecfield$, \ie the problem:
\begin{equation}
\label{eq:zero}
\tag{Opt}
\text{Find $\sol\in\vecspace$ such that $\vecfield(\sol)=0$}.
\end{equation}
This ``vector field formulation'' is the unconstrained case of what is known in the literature as a \acdef{VI} problem \textendash\ see \eg \citet{FP03} for a comprehensive introduction.
In what follows, we will not need the full generality of the \ac{VI} framework and we will develop our results in the context of \eqref{eq:zero} above;
our only blanket assumption in this regard is that the set of solutions $\sols$ of \eqref{eq:zero} is nonempty.%

\paragraph{Feedback assumptions}

Throughout the sequel, we will assume that the optimizer can access $\vecfield$ via a \acdef{SFO}.
This means that at every stage $\run$ of an iterative algorithm, the optimizer can call this black-box mechanism at a point $\current \in \R^{\vdim}$ to get a feedback of the form $\current[\svecfield]= \vecfield(\current) + \current[\noise]$
%
%
where $\current[\noise]\in\dspace$ is an additive noise variable.
Our bare-bones assumptions for this oracle will then be as follows:

\newcommand{\retainlabel}[1]{\label{#1}\sbox0{\ref{#1}}}

\begin{assumption}
 \label{asm:noise}
The noise term $\current[\noise]$ of \ac{SFO} satisfies
\begin{subequations}
\label{eq:noise}
\begin{alignat}{2}
&a)\;\;
	\textit{Zero-mean:}
	&\quad
	&\exof{\current[\noise] \given \current[\filter]}
		= 0.
		\hspace{20em}
	\retainlabel{eq:mean}
	\\
&b)\;\;
	\textit{Variance control:}
	&\quad
	& \exof{\dnorm{\current[\noise]}^2 \given
	\current[\filter]} \leq 
	(\noisedev + \varcontrol \norm{\current-\sol})^2
	~\text{ for all }~ \sol\in\sols.
	\retainlabel{eq:variance}
\end{alignat}
\end{subequations}
where $\noisedev, \varcontrol \ge 0$ and $\current[\filter]$ denotes the history (natural filtration) of $\current$.
\end{assumption}

It is important to note that in \eqref{eq:variance}, $\noisedev$ and $\varcontrol$ play different roles.
When $\varcontrol=0$, the condition corresponds to the classic bounded 
variance assumption on the noise.
At the other end of the spectrum, $\noisedev=0$ implies that the noise vanish on the solution set. This kind of condition has been popularized recently in the machine learning community under the name of interpolation \cite{VMLS+19}.
In the most general case, we have both $\noisedev>0$ and $\varcontrol>0$; then condition \eqref{eq:variance} allows the variance of the noise to exhibit quadratic growth with respect to the distance to the solution set.
For example, for a stochastic oracle of the form $\current[\svecfield]=\svecfield(\xi, \current)$ where $\xi$ is a random variable and $\svecfield$ is a Carathéodory function,\footnote{%
That is, $\svecfield(\xi,\cdot)$ is continuous for almost all $\xi$ and $\svecfield(\cdot,\point)$ is measurable for all $\point$.
} this is trivially satisfied if $\svecfield(\xi,\cdot)$ is Lipschitz and the variance of the noise is bounded on $\sols$.
Therefore, \cref{asm:noise} is fairly weak and verified by most relevant problems.


\section{The \acl{EG} method and its limitations}
\label{sec:EG}

As discussed earlier, the go-to method for saddle-point problems and \aclp{VI} is the \acdef{EG} algorithm of \citet{Kor76} and its variants.
Formally, in the general setting of the previous section, the \ac{EG} algorithm can be stated recursively as:
\begin{equation}
\label{eq:EG}
\tag{EG}
\begin{aligned}
\inter = \current - \current[\step] \current[\svecfield]\;,
\quad\quad
\update = \current - \current[\step] \inter[\svecfield]
\end{aligned}
\end{equation}
where $\current[\step] > 0$ is a variable stepsize sequence.
Heuristically, the basic idea of the method is as follows:
starting from a \emph{base} state $\current$, the algorithm first performs a look-ahead step to generate an intermediate \textendash\ or \emph{leading} \textendash\ state $\inter$;
subsequently, the oracle is called at $\inter$, and the method proceeds to a new state $\update$ by taking a step from the \emph{base} state $\current$.
Hence, the generation of the leading state can be seen as an \emph{exploration} step while the second part is the bona fide \emph{update} step.

One of the reasons for the widespread popularity of \eqref{eq:EG} is that it achieves convergence in all monotone problems, without suffering from the non-convergence phenomena (limit cycles or otherwise) that plague vanilla one-step gradient algorithms \citep{FP03}.
However, this guarantee requires the method to be run with deterministic, \emph{perfect} oracle feedback (\ie $\current[\noise] = 0$ for all $\run$);
if the method is run with genuinely stochastic feedback, the situation is considerably more complicated.

To understand the issues involved, it will be convenient to consider the following elementary example:
\begin{equation}
\label{eq:planar}
\min_{\minvar\in\R}
	\max_{\maxvar\in\R} \,
		\minvar\maxvar.
\end{equation}
Trivially,
the vector field associated to \eqref{eq:planar} is $\vecfield(\minvar,\maxvar) = (\maxvar,-\minvar)$ and the problem's unique solution is $(\sol[\minvar],\sol[\maxvar]) = (0,0)$.
Given the problem's simple structure, one would expect that \eqref{eq:EG} should be easily capable of reaching a solution;
however, as we show below, this is not the case.



\begin{restatable}{proposition}{PropNonConv}
\label{prop:non-conv}
Suppose that \eqref{eq:EG} is run on the problem \eqref{eq:planar} with oracle feedback $\current[\svecfield] = \vecfield(\current[\minvar],\current[\maxvar]) + (\current[\snoise],0)$ for some zero-mean random variable $\current[\snoise]$ with variance $\noisevar>0$.
We then have $\liminf_{\run\to\infty} \exof{\current[\minvar]^{2} + \current[\maxvar]^{2}} > 0$, \ie the iterates of \eqref{eq:EG} remain on average a positive distance away from $0$.
\end{restatable}

Importantly, \cref{prop:non-conv} places \emph{no} restrictions on the algorithm's stepsize sequence and the variance of the noise could be arbitrarily small.
Relegating the details to the appendix, the key to showing this result is the recursion
\makeatletter%
\if@twocolumn%
    \begin{align}
    \ex[\update[\minvar]^{2} + \update[\maxvar]^{2}]
    	&= (1-\current[\step]^{2} + \current[\step]^{4}) \, \ex[\current[\minvar]^{2} + \current[\maxvar]^{2}]
    	\notag\\
    	&+ (1 + \current[\step]^{2}) \current[\step]^{2} \, \noisevar. \notag
    \label{eq:energy-planar}
    \end{align}
\else
    \begin{equation}
    \ex[\update[\minvar]^{2} + \update[\maxvar]^{2}]
    = (1-\current[\step]^{2} + \current[\step]^{4}) \, \ex[\current[\minvar]^{2} + \current[\maxvar]^{2}]
    + (1 + \current[\step]^{2}) \current[\step]^{2} \, \noisevar. \notag
    \label{eq:energy-planar}
    \end{equation}
\fi
from which it follows that $\liminf_{\run} \exof{\current[\minvar]^{2} + \current[\maxvar]^{2}} > 0$.
In turn, this implies that the iterates of \eqref{eq:EG} remain on average a positive distance away from the origin.
This behavior is illustrated clearly in \cref{fig:stoch-bilinear} which shows a typical non-convergent trajectory of \eqref{eq:EG} in the planar problem \eqref{eq:planar}.


\begin{figure}[t]
\centering
\begin{minipage}[b]{.45\textwidth}
	\centering
    \includegraphics[width=.8\linewidth]{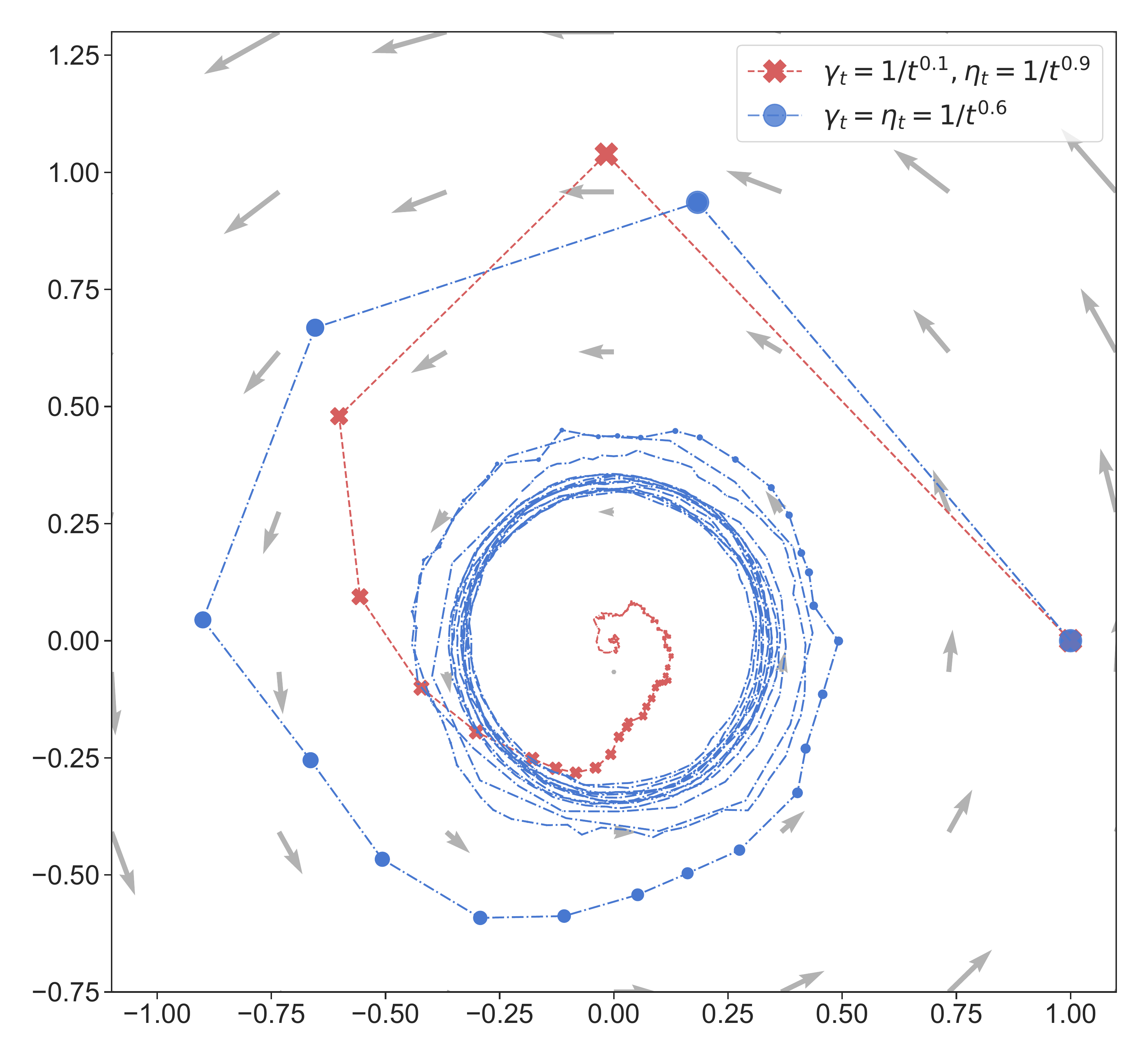}
	\caption[.]{
	Behavior of \eqref{eq:EG} and \eqref{eq:DSEG} on Problem~\eqref{eq:planar} with Gaussian oracle noise.
	Even with a vanishing, square-summable stepsize $\current[\step]=1/\run^{0.6}$, the iterates of \eqref{eq:EG} cycle;
	in contrast, \eqref{eq:DSEG} with $\current[\step]=1/\run^{0.1}$ and $\current[\stepalt]=1/\run^{0.9}$ 
	converges.
	}
	\label{fig:stoch-bilinear}
\end{minipage}
\hspace{1cm}
\begin{minipage}[b]{.45\textwidth}
    \centering

\centering
\begin{tikzpicture}[scale=4]
\clip (-0.2,-0.15) rectangle (1.25, 1.2);

\draw [very thin, gray!15!white] (0,0) grid[step=0.1] (1,1);

    \colorlet{color1}{blue!50}

    \draw (0.75,0.3) node[rotate=-45] {\color{color1!40!black} \tiny $\exponent_\step+\exponent_\step\leq 1$};

\filldraw[thick, draw=color1 , fill=color1 ,  fill opacity=0.1]
    (0,0) -- (0,1)                   
 -- plot [domain=0:1] (\x,1-\x)   
 -- (1,0) -- cycle;

 \colorlet{color2}{red!50}

 \draw (0.75,1.05) node[rotate=0] {\color{color2!80!black} \tiny $\exponent_\stepalt > 1/2$};

\filldraw[dashed, draw=color2 , fill=color2 ,  fill opacity=0.15]
 (0,0.5) -- (1,0.5) -- (1,1) -- (0,1) -- cycle;

 \colorlet{color3}{yellow!80!black}

 \draw (1.05,0.3) node[rotate=-90] {\color{color3!60!black} \tiny $2\exponent_\step+\exponent_\step> 1$};

\filldraw[dashed, draw=color3 , fill=color3 ,  fill opacity=0.15]
(1,1)                   
-- plot [domain=0:0.5] (\x,1-2*\x)   
-- (1,0) -- (1,1) -- cycle;

 \colorlet{color4}{magenta!50}

 \draw (0.15,0.3) node[rotate=90] {\color{color4!70!black} \tiny $\exponent_\step > 1/q$};

\filldraw[dashed, draw=color4 , fill=color4 ,  fill opacity=0.15]
(1,1) -- (1,0) -- (0.2,0) -- (0.2,1) -- cycle; 


\draw[thin,black!50] (0,0) -- (1,1);
 \draw[thin, draw=black!80,<-] (0.25,0.25) to [out=-45,in=-180] (0.3,0.15) node[right] {\tiny\color{black!80} Vanilla EG policies};

 \colorlet{colorMain}{green!70!black}

\filldraw[ultra thick, dashed, draw=colorMain , fill=colorMain ,  fill opacity=0.3]
(0,1)                   
-- plot [domain=0:0.25] (\x,1-2*\x)   
-- (0.5,0.5)  -- cycle; 

\draw[ultra thick, draw=colorMain] (0,1) -- (0.5,0.5);

\draw[thick, draw=colorMain,<-] (0.15,0.85) to [out=40,in=-180] (0.45,0.9) node[right] {\footnotesize \cref{asm:steps}};

\colorlet{colorMain2}{blue!50!green!50!black}

\draw[pattern=horizontal lines, pattern color=colorMain2]
(0.2,0.8) -- (0.2,0.6) -- (0.25,0.5) -- (0.5,0.5) -- cycle; 

\draw[ draw=colorMain2] (0.2,0.8) -- (0.5,0.5);

\draw[thick, draw=colorMain2,<-] (0.27,0.73) to [out=40,in=-180] (0.45,0.8) node[right] {\footnotesize \color{colorMain2} Local results};


 \draw[very thick, ->] (-0.05,0) -- (1.1,0);\draw (1.1,0) node[right] {$\exponent_\step$};\draw [very thick,->] (0,-0.05) -- (0,1.1);\draw (0,1.1) node[above] {$\exponent_\stepalt$};

 \draw[-] (0.02,1) -- (-0.02,1) node[left] {$1$};
 \draw[-] (0.02,0.5) -- (-0.02,0.5) node[left] {$0.5$};
 \draw[-] (1,0.02) -- (1,-0.02) node[below] {$1$};
 \draw[-] (0.5,0.02) -- (0.5,-0.02) node[below] {$0.5$};
 \draw (-0.05,-0.05) node {$0$};

 \draw[-] (0.2,0.02) -- (0.2,-0.02) node[below] {$1/q$};


 \filldraw[thin,black!80] (0,1)  circle[radius=0.3pt]; 
 \draw[thin, draw=black!80,<-] (0.01,1.01) to [out=45,in=-180] (0.15,1.15) node[right] {\tiny\color{black!80} $\mathcal{O}\left(\frac{1}{\run}\right)$ for Affine Operators};

\filldraw[thin,black!80] (0.33333333,0.6666666666)  circle[radius=0.3pt]; 
\draw[thin, draw=black!80,<-] (0.34,0.67) to [out=45,in=-180] (0.5,0.6) node[right] {\tiny\color{black!80}   $\mathcal{O}\left(\frac{1}{\run^{{1}/{3}}}\right)$};

\end{tikzpicture}
    \caption{The stepsize exponents allowed by \cref{asm:steps} for convergence (shaded green).
    Dashed lines are strict frontiers.
    Note that vanilla \ac{EG} (the separatrix $\exponent_{\step} = \exponent_{\stepalt}$) passes just outside of this 
    region, explaining the method's failure.}
    \label{fig:stepsizes}
\end{minipage}
\vskip -\medskipamount
\end{figure}

\section{Extragradient with stepsize scaling}
\label{sec:DSEG}

At a high level, \cref{prop:non-conv} suggests that
the benefit of the exploration step is negated by the noise
as the iterates of \eqref{eq:EG} get closer to the problem's solution set.
To rectify this issue, we will consider a more flexible, \acdef{DSEG} method of the form
\begin{equation}
\label{eq:DSEG}
\tag{DSEG}
\begin{aligned}
\inter = \current - \current[\step] \current[\svecfield],
\quad\quad
\update = \current - \current[\stepalt] \inter[\svecfield],
\end{aligned}
\end{equation}
with $\current[\step]\ge\current[\stepalt]>0$.
The key idea in \eqref{eq:DSEG} is that the scaling of the method's stepsize parameters affords us an extra degree of freedom which can be tuned to order.
In particular, motivated by the failure of \eqref{eq:EG} described in the previous section, we will take a stepsize scaling schedule in which the exploration step evolves at a more aggressive time-scale compared to the update step.
In so doing, the method will keep exploring (possibly with a near-constant stepsize) while maintaining a cautious update policy that does not blindly react to the observed oracle signals.

For illustration and comparison, we plot in \cref{fig:stoch-bilinear} an instance of this method with a fairly aggressive exploration schedule and a respectively conservative update policy.
In contrast to \eqref{eq:EG}, the iterates of \eqref{eq:DSEG} now converge to a solution.
We encode this as a positive counterpart to \cref{prop:non-conv} below:


{
\addtocounter{proposition}{-1}
\renewcommand{\theproposition}{\ref{prop:non-conv}$'$}%
\begin{restatable}{proposition}{PropRemedy}
\label{prop:non-conv2}
Suppose that \eqref{eq:DSEG} is run on the problem \eqref{eq:planar} with oracle feedback $\current[\svecfield] = \vecfield(\current[\minvar],\current[\maxvar]) + (\current[\snoise],0)$ for some zero-mean random variable $\current[\snoise]$ with variance $\noisevar>0$.
If the method's stepsize policies are of the form $\current[\step] = 1/\run^{\pexp}$ and $\current[\stepalt] = 1/\run^{\qexp}$ for some $\qexp > \pexp \geq 0$ with $\pexp+\qexp \leq 1$, we have
$\lim_{\run\to\infty} \exof{\current[\minvar]^{2} + \current[\maxvar]^{2}} \to 0$.
\end{restatable}
}

From an analytic viewpoint, what distinguishes \eqref{eq:EG} from \eqref{eq:DSEG} is the following refined bound:
\smallbreak

\begin{restatable}{lemma}{LemmaDescent}
\label{lem:dsseg-descent-stoch}
Under \cref{asm:Lipschitz,asm:noise}, for all $\run = \running$ and all $\sol \in \sols$, it holds
\begin{align}
    \exof{\norm{\update - \sol}^{2} \given \current[\filter]}
	&\leq (1+\current[\Cons]\varcontrol^2)\norm{\current-\sol}^{2}
	- 2\current[\stepalt] \exof{\product{\vecfield(\inter)}{\inter-\sol} \given \current[\filter]} \notag
	\\
	&
	- \current[\step] \current[\stepalt] (
	    1-\current[\step]^{2}\lips^{2}
	    -8\current[\step]\current[\stepalt]\varcontrol^2)
		\norm{\vecfield(\current)}^{2}
	+ \current[\Cons]\noisevar,
    \label{eq:dsseg-descent-stoch}
\end{align}
with constant $\current[\Cons]=4\current[\step]^2\current[\stepalt]\lips
            + 2\current[\step]^3\current[\stepalt]\lips^2
            + 4\current[\stepalt]^2
            + 16\current[\step]^2\current[\stepalt]^2\varcontrol^2$.
\end{restatable}

The proof of \cref{lem:dsseg-descent-stoch}, which we defer to the supplement, relies on a careful analysis of the update between successive iterates to separate the deterministic and the stochastic effects.
Analyzing the bound of \cref{lem:dsseg-descent-stoch} term-by-term gives a clear picture of how an aggressive exploration stepsize policy can be helpful:%
\begin{itemize}
\setlength{\itemsep}{0pt}
\item
The term $\current[\step] \current[\stepalt] (1-\current[\step]^{2}\lips^{2} -8\current[\step]\current[\stepalt]\varcontrol^2) \norm{\vecfield(\current)}^{2}$ provides a consistently negative contribution as long as $\sup_{\run} \current[\step] < 1/3\max(\smooth, \varcontrol)$.
\item
The term $\current[\Cons]$ 
is antagonistic and needs to be made as small as possible.
\item
The term $\exof{\product{\vecfield(\inter)}{\inter - \sol} \given \current[\filter]}$ plays a lesser role since it is non-negative for variational stable problems (see upcoming \cref{asm:mono}) and is even identically zero in bilinear problems.
\end{itemize}
\vspace{-\smallskipamount}

Therefore, to obtain convergence, one needs the coefficient $\current[\step] \current[\stepalt]$ to be as \emph{large} as possible 
and, concurrently, each of the terms $\current[\step]^{2}\current[\stepalt]$, $ \current[\step]^{3}\current[\stepalt]$, $\current[\stepalt]^{2}$ and $\current[\step]^{2}\current[\stepalt]^2$ that appear in $\current[\Cons]$ 
should be as \emph{small} as possible.
Formally, this would lead to the requirement $\sum_{\run} \current[\step]\current[\stepalt] = \infty$ and $\sum_{\run} \current[\step]^2\current[\stepalt] + \current[\stepalt]^{2} < \infty$.
These conditions can be simultaneously achieved by a suitable choice of $\current[\step]$ and $\current[\stepalt]$ (\cf \cref{prop:non-conv2} above), but they are \emph{mutually exclusive} if $\current[\step] = \current[\stepalt]$.
This observation is the key motivation for the scale separation between the exploration and the update mechanisms in \eqref{eq:DSEG}, and is the principal reason that \eqref{eq:EG} fails to converge in bilinear problems.


\section{Convergence analysis}
\label{sec:results}

We now proceed with our main results for the \ac{DSEG} algorithm.
We begin in \cref{subsec:asym} with an asymptotic convergence analysis for \eqref{eq:DSEG};
subsequently, in \cref{subsec:rate}, we examine the algorithm's rate of convergence;
finally, in \cref{subsec:affine}, we zero in on affine problems.
Given our interest in non-monotone problems, we make a clear distinction between global results (which require global assumptions) and local ones (which apply to more general problems).

\subsection{Asymptotic convergence}
\label{subsec:asym}

\para{Global convergence}

\!\!Our 
assumption 
for global convergence is a 
variational stability condition.

\begin{assumption}
\label{asm:mono}
The operator $\vecfield$ satisfies $\product{\vecfield(\point)}{\point-\sol}\ge0$ for all $\point\in\vecspace$, $\sol\in\sols$.
\end{assumption}

\Cref{asm:mono} is verified for all monotone operators
but it also encompasses a wide range of non-monotone problems;
for an overview see \eg\cite{FP03,IJOT17,KS19,LMRZ+20,MLZF+19} and references therein.

To leverage this assumption, we will further need the algorithm's update step to decrease sufficiently quickly relative to the corresponding exploration step.
Formally (and with a fair degree of hindsight), this boils down to the following:

\begin{assumption}
\label{asm:steps}
The stepsizes of \eqref{eq:DSEG} satisfy
$\sum_{\run}\current[\step]\current[\stepalt] = \infty$,\;
$\sum_{\run}\current[\stepalt]^{2} < \infty $,\;
and
$\sum_{\run}\current[\step]^2 \current[\stepalt] < \infty$.
\end{assumption}

\Cref{asm:steps} essentially posits that $\current[\stepalt]/\current[\step] \to 0$ as $\run\to\infty$, so it reflects precisely the principle of ``aggressive exploration, conservative updates''.
In particular, \cref{asm:steps} rules out the choice $\current[\step] = \current[\stepalt]$ which would yield the vanilla \ac{EG} algorithm, providing further evidence for the use of a double stepsize policy.
A typical stepsize policy for \eqref{eq:DSEG} is 
\begin{align}
\label{eq:stepsizes}
    \current[\step]=\frac{\step}{(\run+\stepoffset)^{\exponent_{\step}}} ~~~ \text{ and  } ~~~ \current[\stepalt]= \frac{\stepalt}{(\run+\stepoffset)^{\exponent_{\stepalt}}}
\end{align}
for some $\gamma,\eta,\stepoffset>0$ and exponents $\exponent_{\step},\exponent_{\stepalt}\in [0,1]$. \cref{asm:steps} then translates as $\exponent_{\step}+\exponent_{\stepalt} \leq 1$, $2\exponent_{\stepalt}>1$, and $2\exponent_{\step}+\exponent_{\stepalt}>1 $ as represented in \cref{fig:stepsizes}.
With this in mind, we have the following convergence result.

\begin{restatable}{theorem}{ThmGlobalAs}
\label{thm:dsseg-general-as}
Let \cref{asm:Lipschitz,asm:noise,asm:mono,asm:steps} hold and $\sup_{\run} \current[\step] < 1/3\max(\smooth, \varcontrol)$, then the iterates $\current$ of $\eqref{eq:DSEG}$ converge almost surely to a solution $\sol$ of \eqref{eq:zero}.
\end{restatable}


As far as we are aware, this is the first result of this type for stochastic first-order methods:
almost sure convergence typically requires stronger hypotheses guaranteeing that $\product{\vecfield(\point)}{\point-\sol}$ is uniformly positive when $\point \notin \sols$ \cite{KS19,MLZF+19}.
In particular, \cref{thm:dsseg-general-as} implies the almost sure convergence of the algorithm for bilinear problems like \eqref{eq:planar} where \ac{EG} and standard gradient methods do not converge.

\para{Local convergence}

To extend \cref{thm:dsseg-general-as} to fully non-monotone settings, we will consider the following local
version of \cref{asm:Lipschitz,asm:noise,asm:mono} near a solution point $\sol$:

\smallbreak
\asmtag{\ref*{asm:Lipschitz}$'$}
\begin{restatable}{assumption}{AsmLipsLoc}
\label{asm:Lipschitz-loc}
The field $\vecfield$ is $\lips$-Lipschitz continuous near $\sol$, \ie for all $\point,\pointalt$ near $\sol$,
\begin{equation}
\notag
\dnorm{\vecfield(\pointalt) - \vecfield(\point)}
	\leq \lips \norm{\pointalt - \point}.
\label{eq:Lipschitz-loc}
\end{equation}
\end{restatable}

\smallbreak
\asmtag{\ref*{asm:noise}$'$}
\begin{restatable}{assumption}{AsmNoiseLoc}
\label{asm:noise-loc}
Let $\sol\in\sols$ and $\nhd$ be a neighborhood of $\sol$. The noise term $\current[\noise]$ of \ac{SFO} satisfies
\begin{subequations}
\label{eq:noise-loc}
\begin{alignat}{2}
&a)\;\;
	\textit{Zero-mean:}
	&\quad
	&\exof{\current[\noise] \given \current[\filter]}\one_{\{\current\in\nhd\}}
	= 0.
	\hspace{17em}
	\retainlabel{eq:mean-loc}
	\\
&b)\;\;
	\textit{Moment control:}
	&\quad
	&\exof{\dnorm{\current[\noise]}^{\probmoment} \given
	\current[\filter]}\one_{\{\current\in\nhd\}} \leq 
	(\noisedev + \varcontrol \norm{\current-\sol})^{\probmoment}.
	\retainlabel{eq:variance-loc}
\end{alignat}
\end{subequations}
for some $\probmoment > 2$ and $\noisedev, \varcontrol \ge 0$.
\end{restatable}

\smallbreak
\asmtag{\ref*{asm:mono}$'$}
\begin{restatable}{assumption}{AsmMonoLoc}
\label{asm:mono-loc}
The operator $\vecfield$ satisfies $\product{\vecfield(\point)}{\point-\sol}\ge0$ for all $\point$ near $\sol$.
\end{restatable}
\addtocounter{assumption}{-3}

Notice that \eqref{eq:variance-loc} is slightly stronger than \eqref{eq:variance} in the sense that we now require to control the $\probmoment^{th}$ moment of the noise for some $\probmoment>2$.
Nonetheless, this condition as well as the unbiasedness assumption only need to be satisfied in a neighborhood of $\sol$.
Our next result shows that, with these modified assumptions, the \ac{DSEG} algorithm converges locally to solutions with high probability:
\smallbreak




\begin{restatable}{theorem}{ThmLocalAs}
\label{thm:dsseg-local-as}
Fix a tolerance level $\delta>0$ and suppose that \cref{asm:Lipschitz-loc,asm:noise-loc,asm:mono-loc} hold for some isolated solution $\sol$ of \eqref{eq:zero}.
Assume further that \eqref{eq:DSEG} is run with stepsize parameters of the form \eqref{eq:stepsizes} with small enough $\step$, $\stepalt$ and proper choice of $\exponent_{\step}, \exponent_{\stepalt}$ \textpar{\cf \cref{fig:stepsizes}}.
If the algorithm is not initialized too far from $\sol$, its iterates converge to $\sol$ with probability at least $1-\delta$.
\end{restatable}

The first step towards proving \cref{thm:dsseg-local-as} is to show that the generated iterates stay close to $\sol$ with arbitrarily high probability.
To achieve this, one needs to control the total noise accumulating from each noisy step, a task which is made difficult by the fact that the norm of the \ac{SFO} feedback can only be upper bounded recursively and thus depends on previous iterates.
In the supplement, we dedicate a lemma to the study of such recursive stochastic processes, and we build our analysis on 
this lemma.
\subsection{Convergence rates}
\label{subsec:rate}

\para{Global rate}

To study the algorithm's convergence rate, we will require the following error bound condition:

\begin{assumption}
\label{asm:err-bound}
For some $\errbcst>0$ and all $\point\in\vecspace$, we have 
\begin{equation}
\label{eq:err-bound}
\dnorm{\vecfield(\point)}
	\geq \errbcst \dist(\point, \sols).
	\tag{EB}
\end{equation}
\end{assumption}

This kind of error bound is standard in the literature on \aclp{VI} for deriving last iterate convergence rates
\citep[see \eg][]{luo1993error,Tse95,FP03,Sol03,Mal19}.
In particular,  \cref{asm:err-bound} is satisfied by
\vspace*{-\smallskipamount}
\begin{enumerate}
[\itshape a\upshape)]
\setlength{\itemsep}{0pt}
	\item \emph{Strongly monotone operators:}
	here, $\tau$ is the strong monotonicity modulus.
	\item \emph{Affine operators:}
	for $\vecfield(\point)=\mat\point+\vvec$ where $\mat$ is a matrix of size $\vdim\times\vdim$ and $\vvec$ is a $\vdim$-dimensional vector, $\tau$ is the minimum non-zero singular value of $\mat$.
\end{enumerate}
\vspace*{-\smallskipamount}
In this sense, \cref{asm:err-bound} provides a unified umbrella for two types of problems that are typically considered to be poles apart. 
Our first result in this context is as follows:
\smallbreak

\begin{theorem}
\label{thm:dsseg-general-rate}
Suppose that \cref{asm:Lipschitz,asm:noise,asm:mono,asm:err-bound} hold and assume 
that $\current[\step]\le\cons/\lips$ with $\cons<1$.
Then:
\begin{enumerate}
\item\label{thm:dsseg-general-rate-a}
If \eqref{eq:DSEG} is run with $\current[\step] \equiv \step$, $\current[\stepalt] \equiv \stepalt$, we have:
\begin{equation}
\exof{\dist(\current, \sols)^2}
	\leq (1-\minuscst)^{\run-1}\dist(\state_\start, \sols)^2
		+ \frac{\pluscst}{\minuscst} \notag
\end{equation}
with constants $\pluscst=(2\step^2\stepalt\lips+\step^3\stepalt\lips^2+\stepalt^2)\noisevar$ and $\minuscst=\step\stepalt\errbcst^2(1-\cons^2)$.%
\footnote{For better readability, these constants are stated for the case $\varcontrol=0$. 
On the other hand, if $\noisedev=0$ (and $\varcontrol\geq 0$), a geometric convergence can be proved. The same arguments apply to \cref{thm:dsseg-affine}.}
\item\label{thm:dsseg-general-rate-b}
If \eqref{eq:DSEG} is run with 
$\current[\step] = \step/(\run+\stepoffset)^{1-\exponentalt}$ and $\current[\stepalt] = \stepalt/(\run+\stepoffset)^{\exponentalt}$ for some $\exponentalt\in(1/2,1)$, we have:
\begin{equation}
\exof{\dist(\current, \sols)^{2}}
	\leq
	\frac{\pluscst}{\minuscst-\exponent} \frac{1}{\run^\exponent}
	+ \smalloh\left(\frac{1}{\run^\exponent}\right) \notag
\end{equation}
where $\exponent = \min(1-\exponentalt, 2\exponentalt-1)$ and we further assume that $\step\stepalt\errbcst^2(1-\cons^2) > \exponent$.
In particular, the optimal rate is attained when $\exponentalt=2/3$, which gives $\ex[\dist(\current, \sols)^{2}]=\bigoh(1/\run^{1/3})$.
\end{enumerate}
\end{theorem}

The first part of \cref{thm:dsseg-general-rate} shows that, if \eqref{eq:DSEG} is run with constant stepsizes, the initial condition is forgotten exponentially fast and the iterates converge to a neighborhood of $\sols$ (though, in line with previous results, convergence cannot be achieved in this case). 
To make this neighborhood small, we need to decrease both $\step$ and $\stepalt/\step$;
this would be impossible for vanilla \eqref{eq:EG} for which $\stepalt/\step=1$.

The second part of \cref{thm:dsseg-general-rate} provides an $\bigoh(1/\run^{1/3})$ last-iterate convergence rate.
In \cref{subsec:affine}, we further improve this rate
to $\bigoh(1/\run)$ for affine operators by exploiting their particular structure. 

\para{Local rate}

To study the algorithm's local rate of convergence, we will focus on solutions of \eqref{eq:zero} that satisfy the following Jacobian regularity condition:
\asmtag{\ref*{asm:err-bound}$'$}
\begin{restatable}{assumption}{AsmErrLoc}
\label{asm:err-bound-loc}
$\vecfield$ is differentiable at $\sol$ and its Jacobian matrix $\Jacf{\vecfield}{\sol}$ is invertible.
\end{restatable}

The link between \cref{asm:err-bound,asm:err-bound-loc} is provided by the following proposition:

\begin{proposition}
\label{prop:local-err-bound}
If a solution $\sol$ satisfies \cref{asm:err-bound-loc}, it satisfies \eqref{eq:err-bound} in a neighborhood of $\sol$.
\end{proposition}

The proof of \cref{prop:local-err-bound} follows by performing a Taylor expansion of $\vecfield$ and invoking the minimax characterization of the singular values of a matrix;
we give the details in the supplement.
For our purposes, what is more important is that \eqref{eq:err-bound} has now been reduced to a \emph{pointwise} condition;
under this much lighter requirement, we have:
\smallbreak

\begin{restatable}{theorem}{ThmLocalRate}
\label{thm:dsseg-local}
Fix a tolerance level $\delta>0$ and suppose that \cref{asm:Lipschitz-loc,asm:noise-loc,asm:mono-loc} and \ref{asm:err-bound-loc} hold for some isolated solution $\sol$ of \eqref{eq:zero} with $\probmoment>3$.
Assume further $\sol$ satisfies \cref{asm:err-bound-loc} and \eqref{eq:DSEG} is run with stepsize parameters of the form
$\current[\step]=\step/(\run+\stepoffset)^{1/3}$ and
$\current[\stepalt]=\stepalt/(\run+\stepoffset)^{2/3}$
with large enough $\stepoffset,\stepalt > 0$.
Then, there exist neighborhoods $\nhd$, $\nhdalt$ of $\sol$ and an event $\eventI_{\nhd}$
such that:
\begin{enumerate}
[\itshape a\upshape)]
\setlength{\itemsep}{0pt}
\item \label{dsseg-local-a}
$\probof{\eventI_{\nhd} \given \state_{\start} \in \nhd} \ge 1-\smallproba$.
\item \label{dsseg-local-b}
$\probof{\current \in \nhdalt \; \text{for all $\run$} \given \eventI_{\nhd}} = 1$.
\item \label{dsseg-local-c}
$\exof{\norm{\state_{\nRunsNew} - \sol}^{2} \given \eventI_{\nhd}} = \bigoh\parens*{1/\run^{1/3}}$
\end{enumerate}
In words, if \eqref{eq:DSEG} is not initialized too far from $\sol$, the iterates $\current$ remain close to $\sol$ with probability at least $1-\delta$ and, conditioned on this event, $\current$ converges to $\sol$ at a rate $\bigoh(1/\run^{1/3})$ in mean square error.
\end{restatable}

Taken together, \cref{thm:dsseg-general-as,thm:dsseg-local} show that for all monotone stochastic problems with a non-degenerate critical point, employing the suggested stepsize policy yields an asymptotic $\bigoh(1/\run^{1/3})$ rate.
In more detail, the last point of \cref{thm:dsseg-local} shows that, with the same kind of stepsizes as in the second part of \cref{thm:dsseg-general-rate}, we can retrieve a $\bigoh(1/\run^{1/3})$ convergence rate provided that the iterates stay close to the solution. Note that this rate is not a localization of \cref{thm:dsseg-general-rate} because, after conditioning,
\emph{the unbiasedness of the noise is not guaranteed}.
To overcome this issue, our proof draws inspiration from \citet{HIMM19} but the use of double stepsizes requires a much more intricate analysis which is reflected in the stronger noise assumption.
\subsection{A case study of affine operators}
\label{subsec:affine}

We terminate our analysis with a dedicated treatment of affine operators
which are commonly studied as a first step to understand the training of \acp{GAN}
\citep{DISZ18,GBVV+19,MNG18,LS19,ZY20,ASMSG20}. 
The following result improves the $\bigoh(1/\run^{1/3})$ rate of \cref{thm:dsseg-general-rate} to $\bigoh(1/\run)$ for 
affine operators.
\smallbreak

\begin{theorem}\label{thm:dsseg-affine}
Let $\vecfield$ be an affine operator satisfying \cref{asm:mono}, and suppose that \cref{asm:noise} holds.
Take a constant exploration stepsize  $\current[\step] \equiv \step \le\cons/\lips$ with $\cons<1$ (here $\lips$ is the largest singular value of the associated matrix). Then, the iterates $\seqinf{\state}{\run}$ of \eqref{eq:DSEG} enjoy the following rates:
\begin{enumerate}
\item \label{thm:dsseg-affine-a}
If the update stepsize is constant $\current[\stepalt]\equiv\stepalt\le\step$, then:
\begin{align}
    \ex[\dist(\current, \sols)^2]
    \leq
	(1-\minuscst)^{\run-1}\dist(\state_\start, \sols)^2
	+ \frac{\pluscst}{\minuscst} \notag
\end{align}
with  $\pluscst=\stepalt^2(1+\cons^2)\noisevar$ and $\minuscst=\step\stepalt\errbcst^2(1-\cons^2)$.
\item \label{thm:dsseg-affine-b}
If the update stepsize is of the form $\current[\stepalt]= \stepalt/(\run+\stepoffset)$
for $\stepalt>1/(\errbcst^2\step(1-\cons^2))$ and $\stepoffset>\stepalt/\step$, then:
\begin{equation}
    \ex[\dist(\current, \sols)^{2}]
    \leq
	\frac{\pluscst}{\minuscst-1} \frac{1}{\run}
	+ \smalloh\left(\frac{1}{\run}\right). \notag
\end{equation}
\end{enumerate}
\end{theorem}
The proof of this theorem relies on the derivation of another descent lemma similar to \cref{lem:dsseg-descent-stoch} but tailored to affine operators.
Note also that \cref{asm:Lipschitz,asm:err-bound} are automatically verified in this case.

\cref{thm:dsseg-affine} mirrors \cref{thm:dsseg-general-rate}; however, in Part \ref{thm:dsseg-affine-a}  of \cref{thm:dsseg-affine}, the final precision is only determined by $\noisevar$ and $\stepalt/\step$. Thus, compared to \cref{thm:dsseg-general-rate}, there is no need to decrease $\step$ to obtain an arbitrarily high accuracy solution. The weaker dependence on $\step$ is further confirmed by Part \ref{thm:dsseg-affine-b}, which shows a $\bigoh(1/\run)$ rate with $\current[\step]$ constant.
As far as we are aware, this result gives the best convergence rate for stochastic affine operators compared to the literature, and it gives yet another motivation for the use of a
double stepsize strategy.



\section{Numerical experiments}
\label{sec:experiments}


\begin{figure*}[t]
    \centering
    \includegraphics[width=0.32\linewidth]{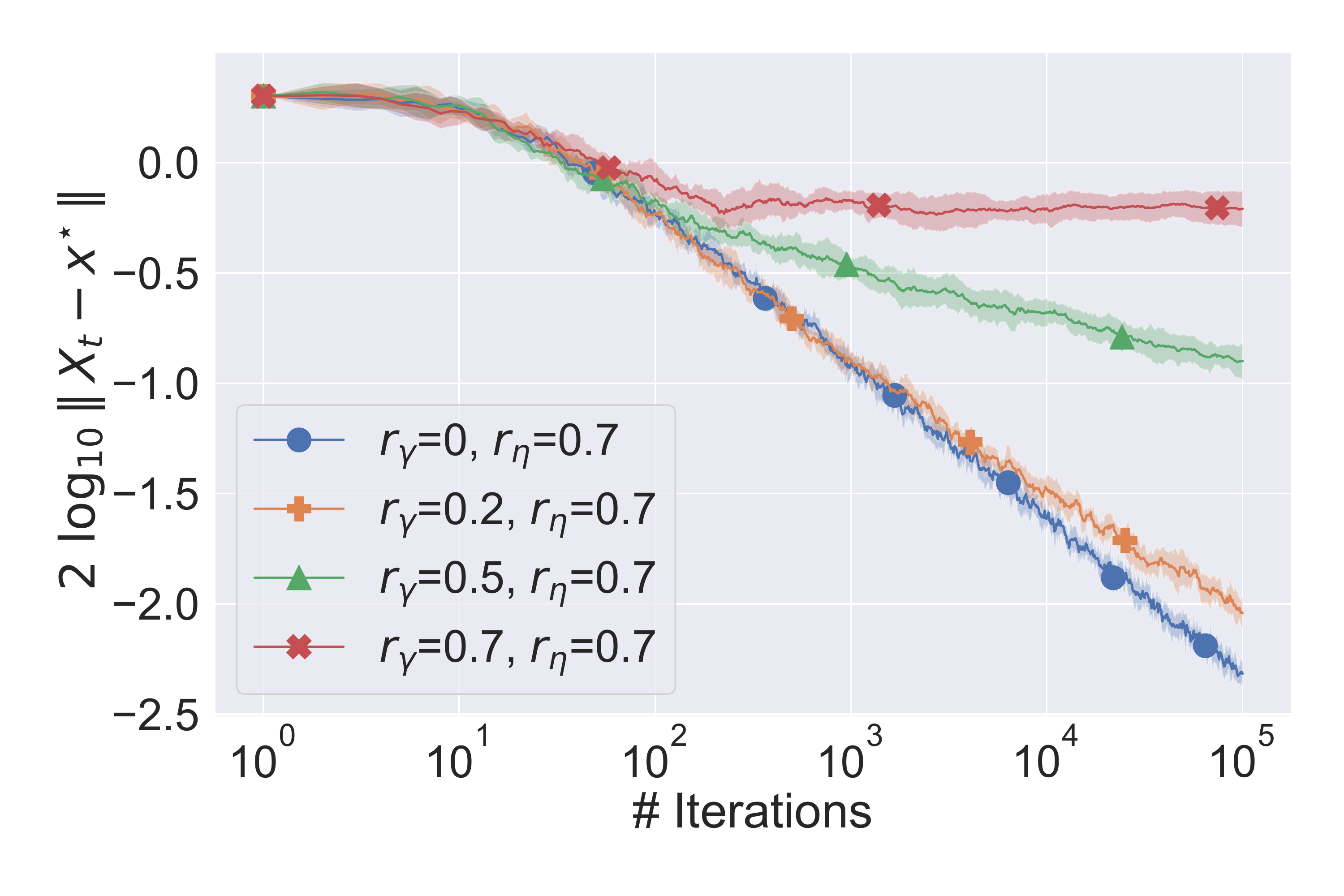}
    \includegraphics[width=0.32\linewidth]{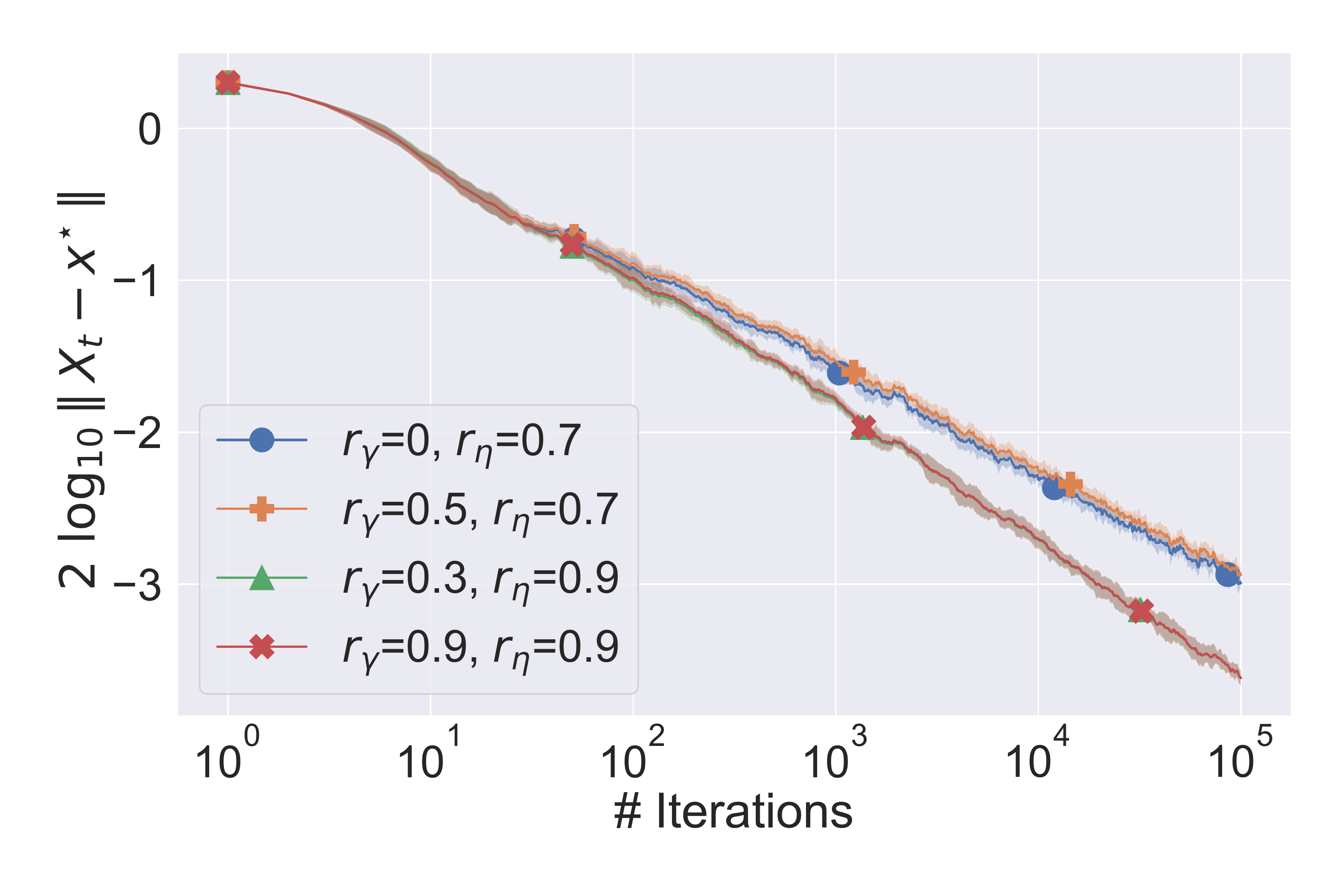}
    \includegraphics[width=0.32\linewidth]{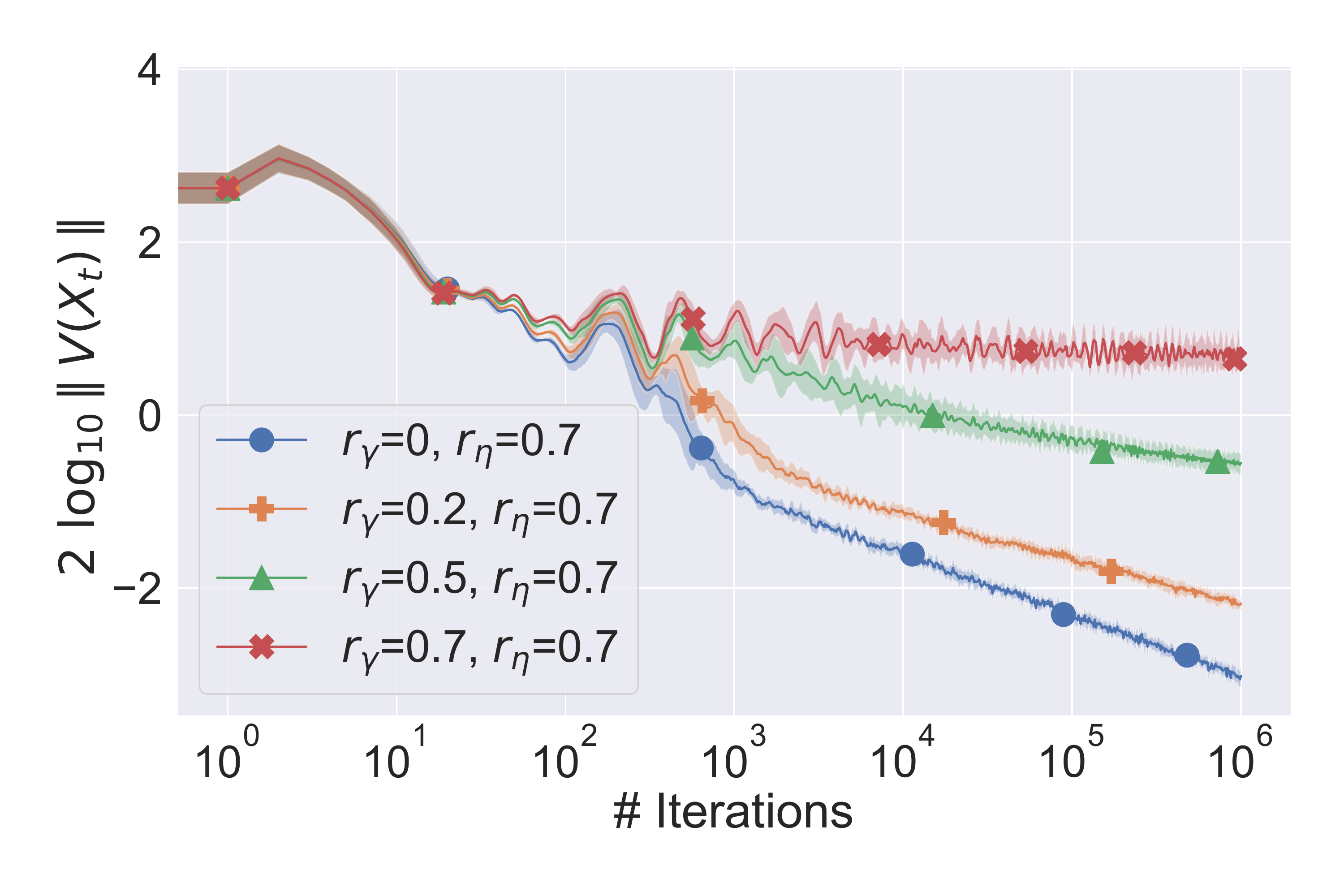}
    \caption{Convergence of a \eqref{eq:DSEG} scheme in stochastic bilinear (left), strongly convex-concave (middle) and non convex-concave linear quadratic Gaussian GAN (right) problems.
    All curves are averaged over $10$ runs with the shaded area indicating the standard deviation.
    The benefit of aggressive exploration is evident.}
    \label{fig:illustration}
    \vskip -\medskipamount
\end{figure*}


This section investigates numerically the benefits of double stepsizes. 
We run \eqref{eq:DSEG} with stepsize of the form \eqref{eq:stepsizes} on three different problems:
\emph{i)} a bilinear zero-sum game, \emph{ii)} a strongly convex-concave game and \emph{iii)} a non convex-concave linear quadratic Gaussian GAN model \cite{DISZ18,NK17}. We examine their behavior when $\exponent_{\step}$ and $\exponent_{\stepalt}$ vary.
The exact description of the problems and the experimental details are deferred to the supplement.

As shown in \cref{fig:illustration}, for bilinear game and Gaussian GAN examples, choosing $\exponent_{\stepalt}<\exponent_{\step}$ turns out to be necessary for the convergence of the algorithm, and the convergence speed is positively related to the difference $\exponent_{\step}-\exponent_{\stepalt}$, as per our analysis.
For a strongly convex-concave problem, it is known that the iterates produced by \eqref{eq:EG} with noisy feedback achieve $\bigoh(1/\run)$ convergence for proper choice of $\seqinf{\step}{\run}$ \cite{KS19,HIMM19}.
Our experiment moreover reveals that when a double step-size policy is considered, the convergence speed of the algorithm seems to only depend on $\seqinf{\stepalt}{\run}$ and using aggressive $\seqinf{\step}{\run}$ has little influence, if any, suggesting that taking a larger exploration step may be a universal solution.
Going one step further, we conduct experiments and observe similar phenomena for the generalized \acl{OG} method \citep{MOP20,RYY19} when the output vector is appropriately chosen.
We refer the interested reader to the supplement for a dedicated discussion.

\section{Conclusion}
\label{sec:conclusion}

In this paper, we examined the benefits of employing a \acli{DSEG} method for which the exploration step is more aggressive than the update step.
This additional flexibility turns out to be both necessary and sufficient for the method to achieve superior convergence properties relative to vanilla stochastic \acl{EG} methods in a large spectrum of problems including bilinear games and some non convex-concave models.

Our results constitute a first attempt towards designing an algorithm that provably avoids cycles and similar non-convergent phenomena in a fully stochastic setting.
Several interesting future directions include an extended analysis with relaxation of the variational stability assumption as well as the design of a fully adaptive and/or universal method on the basis of our results.

\section*{Broader impact}
\label{sec:broad}
This work does not present any foreseeable societal consequence.

\section*{Acknowledgments}
This work has been partially supported by MIAI@Grenoble Alpes, (ANR-19-P3IA-0003).

\bibliographystyle{icml2020}
\bibliography{References/biblio}

\newpage
\appendix
\numberwithin{equation}{section}		
\numberwithin{lemma}{section}		
\numberwithin{proposition}{section}		
\numberwithin{theorem}{section}		

\section{Additional related work}
\label{app:related}
The first analysis of \acdef{EG} with stochastic feedback traces back to the work of \citet{JNT11}, where a $\bigoh(1/\sqrt{\run})$ ergodic convergence was shown for monotone problems, and this rate is known to be optimal without further assumptions \cite{NY83}.%
\footnote{Precisely, the results of \cite{JNT11,KS19,MLZF+19} concern the more general \acl{MP} algorithm, which generalized \acl{EG} to the Bregman setting.}
Since then, a large number of works have been dedicated to studying the convergence behavior of stochastic \ac{EG}-type algorithms, either for better understanding of the algorithm itself or in the hope of finding a better way to incorporate \ac{EG} with stochasticity.

Almost sure convergence of stochastic \ac{EG} was first investigated in \citet{KS19}. In the said paper, almost convergence was shown for \emph{pseudomonotone plus} operators and by additionally assuming that the map is \emph{strongly pseudomonotone} or \emph{monotone and weak-sharp}, the authors managed to prove a $\bigoh(1/{\run})$ convergence of the iterate produced by the algorithm.
In \cite{MLZF+19}, the pseudo-monotonicity-plus assumption is relaxed to show that stochastic \ac{EG} still enjoys last-iterate convergence in \emph{strict coherent} problems.
Nonetheless, these results fail to justify the use of \ac{EG} for stochastic monotone problems, as illustrated in \cref{sec:EG}.
Therefore, to improve the convergence behavior of \ac{EG} in stochastic problems, several modifications to the original stochastic \ac{EG} have been proposed \cite{CGFLJ19,IJOT17,MKSR+20}.
In addition to the ones discussed in \cref{sec:intro}, \citet{MKSR+20} advocated a repeated sampling strategy and illustrated numerically its better performance when applied to GAN training. They also showed that their proposed algorithm retain the same convergence guarantee as traditional stochastic \ac{EG}.

In order to reduce the overall computational cost, another line of research aims at designing optimization methods that solve variational problems with a single oracle call per iteration (instead of the two in \ac{EG}).
Algorithms of this family include for example \acdef{OG} \cite{DISZ18} and \acdef{PEG} \cite{GBVV+19,Pop80}.
See \citet{HIMM19} for a recent overview and corresponding treatment in the stochastic setting.
Very recently, the convergence of stochastic \ac{OG} are further improved in two different ways.
In \cite{FOP20}, the authors introduced a multistage version of \ac{OG} for stochastic strongly monotone problems to optimize the dependence of convergence speed on initial error and noise characteristics.
On the other hand, inspired by the success of adaptive methods in deep learning, \citet{LMRZ+20} designed an adaptive variant of \ac{OG} and showed that it enjoyed an adaptive complexity that varies according to the growth rate of the cumulative stochastic gradient.
To complete the list, also in the goal of reducing overall computation though under a quite different perspective, \citet{SCDAJ19} analyzed a randomized version of stochastic \ac{EG} in multiplayer game to make the extrapolation step amenable to massive multiplayer settings.

\section{Generalized optimistic gradient}
\label{app:beyond-eg}


Considering the similarity between \ac{EG} and its single-call variants, we believe our analysis on \eqref{eq:DSEG} also suggests essential modifications in terms of stepsizes that should be carried out for these algorithms in the face of stochasticity.
As an example, we investigate the \ac{OG} method of \citet{DISZ18}, and find out that some surprising conclusions can be drawn after applying the double stepsize rule.
The generalized \ac{OG} recursion is commonly stated as follows \citep{MOP20,RYY19}:
\begin{equation} \label{eq:OG}
    \tag{OG}
    \update = \current - \current[\stepalt] \current[\svecfield] - \current[\step] (\current[\svecfield]-\last[\svecfield])
\end{equation}
where $\current[\step]$ is sometimes called the \emph{optimism} rate. Similarly to our conclusions, it has been empirically observed that taking large optimism rate often yields better convergence in stochastic problems \cite{PDZC19}.

\citet{HIMM19} pointed out that  \ac{OG} is equivalent to the modified Arrow-Hurwitz method introduced by \citet{Pop80} and also referred to as \ac{PEG} by \citet{GBVV+19}. Using a double stepsize policy,  \ac{PEG} becomes:
\begin{equation}
\label{eq:DSPEG}
\tag{DSPEG}
\begin{aligned}
\inter
	= \current - \current[\step] \past[\svecfield],
	\quad
\update
	= \current - \current[\stepalt] \inter[\svecfield].
\end{aligned}
\end{equation}
Hence, leading states can be recursively written as
\begin{equation}
    \inter = \past - \pastpast[\stepalt] \past[\svecfield] - \past[\step] \past [\svecfield] + \pastpast[\step]\pastpast[\svecfield]. \notag
\end{equation}
We thereby see that \eqref{eq:OG} and \eqref{eq:DSPEG} are almost equivalent and they mostly differ in the choice of vectors that the method outputs at the end:
\ac{OG} suggests outputting $\current$ while \ac{PEG} instead looks at $\current+\last[\step]\last[\svecfield]$.
This nuance turns out to be of importance when generalized \ac{OG} is applied to stochastic problems.
By analogy with our analysis for \eqref{eq:DSEG}, we reasonably conjecture that taking $\current[\stepalt]<\current[\step]$ guarantees the convergence of $\current+\last[\step]\last[\svecfield]$, and this may occur even if $\current[\step]$ is set to constant.
Nonetheless, this also implies that if the noise is not vanishing at the solution, $\current$, which corresponds to the exploration state in \ac{PEG}, might exhibit much slower convergence or even not converge at all.

To summarize, when running \eqref{eq:OG} for stochastic problems,
we should look at the \emph{residual iterate} $\current+\last[\step]\last[\svecfield]$ instead of the \emph{optimistic iterate} $\current$.
Interestingly, this conclusion is consistent with the ODE analysis of \ac{OG} by \citet{RYY19}, and explains some experimental results of said work. Furthermore, taking an aggressive exploration step $\current[\step]$ and a more conservative update step $\current[\stepalt]$ may be very beneficial both in theory (for the last iterate convergence and rate) and in practice as confirmed by our experiments just below.

\section{Experimental details and additional experiments}
\label{app:exps}

\begin{figure*}[t]
    \centering
    \begin{subfigure}[b]{\linewidth}
    \includegraphics[width=0.33\linewidth]{Figures/bilinear_EG_fix_eta}
    \includegraphics[width=0.33\linewidth]{Figures/strong_EG}
    \includegraphics[width=0.33\linewidth]{Figures/lc_EG_fix_eta}
    \end{subfigure}
    \begin{subfigure}[b]{\linewidth}
    \includegraphics[width=0.33\linewidth]{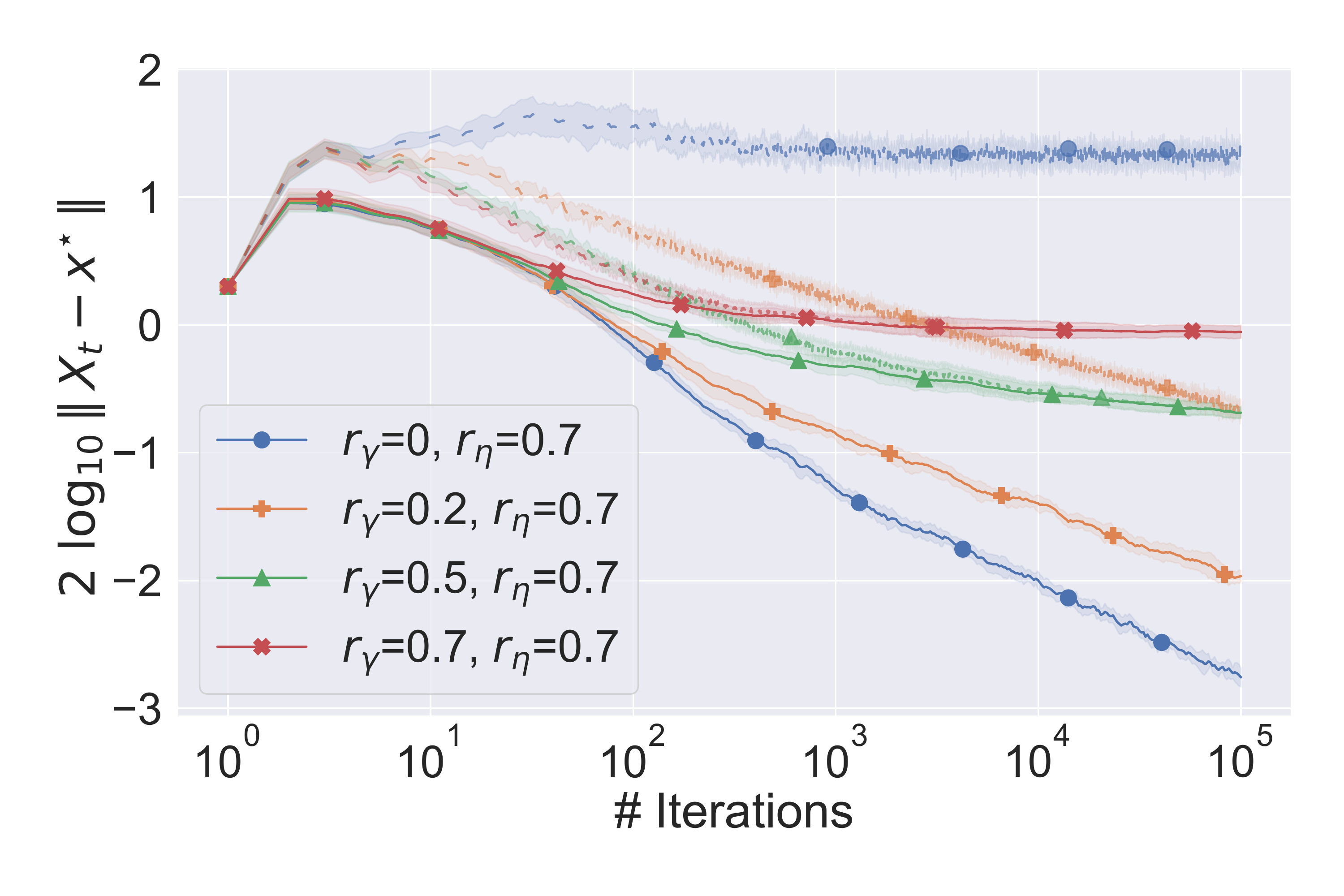}
    \includegraphics[width=0.33\linewidth]{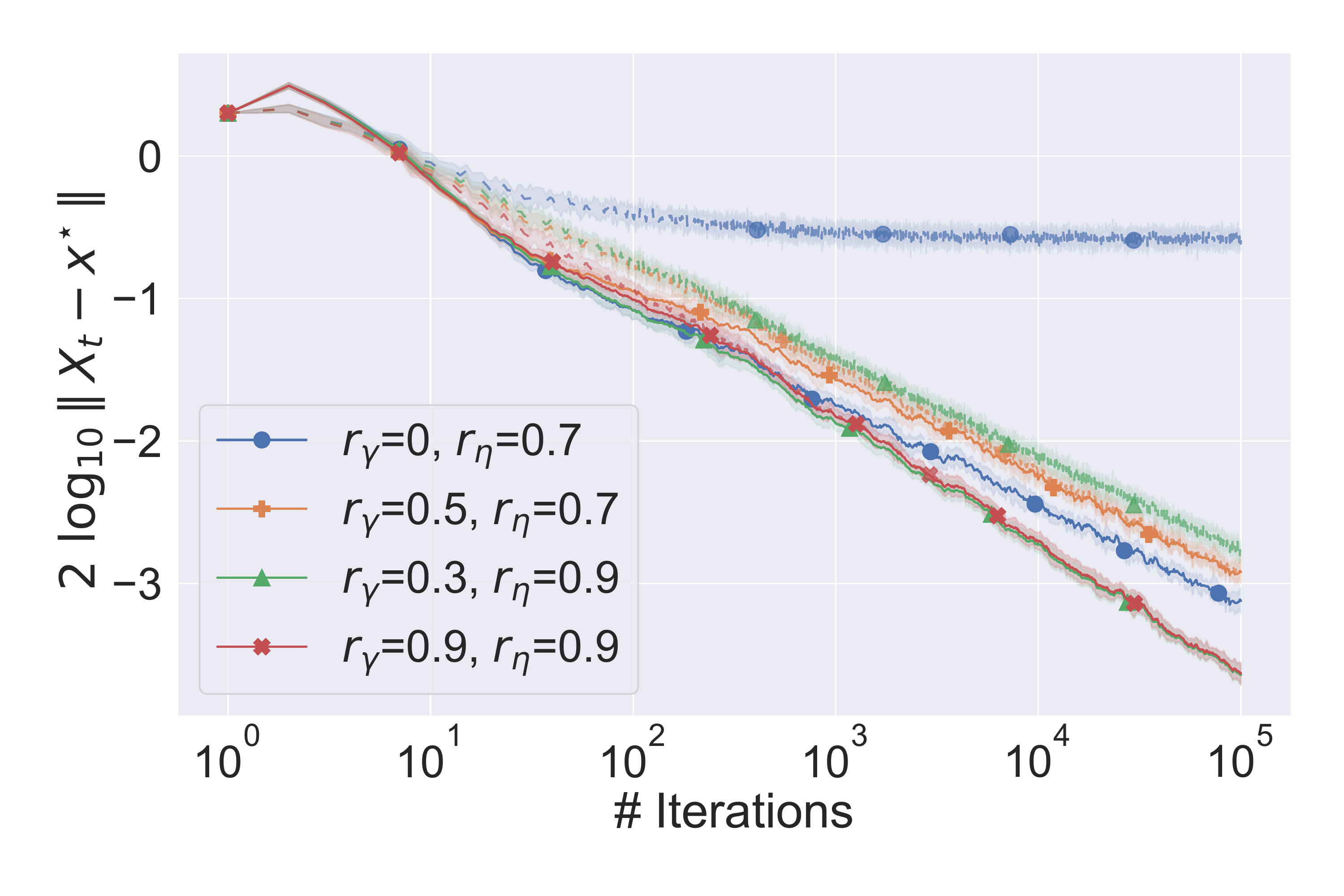}
    \includegraphics[width=0.33\linewidth]{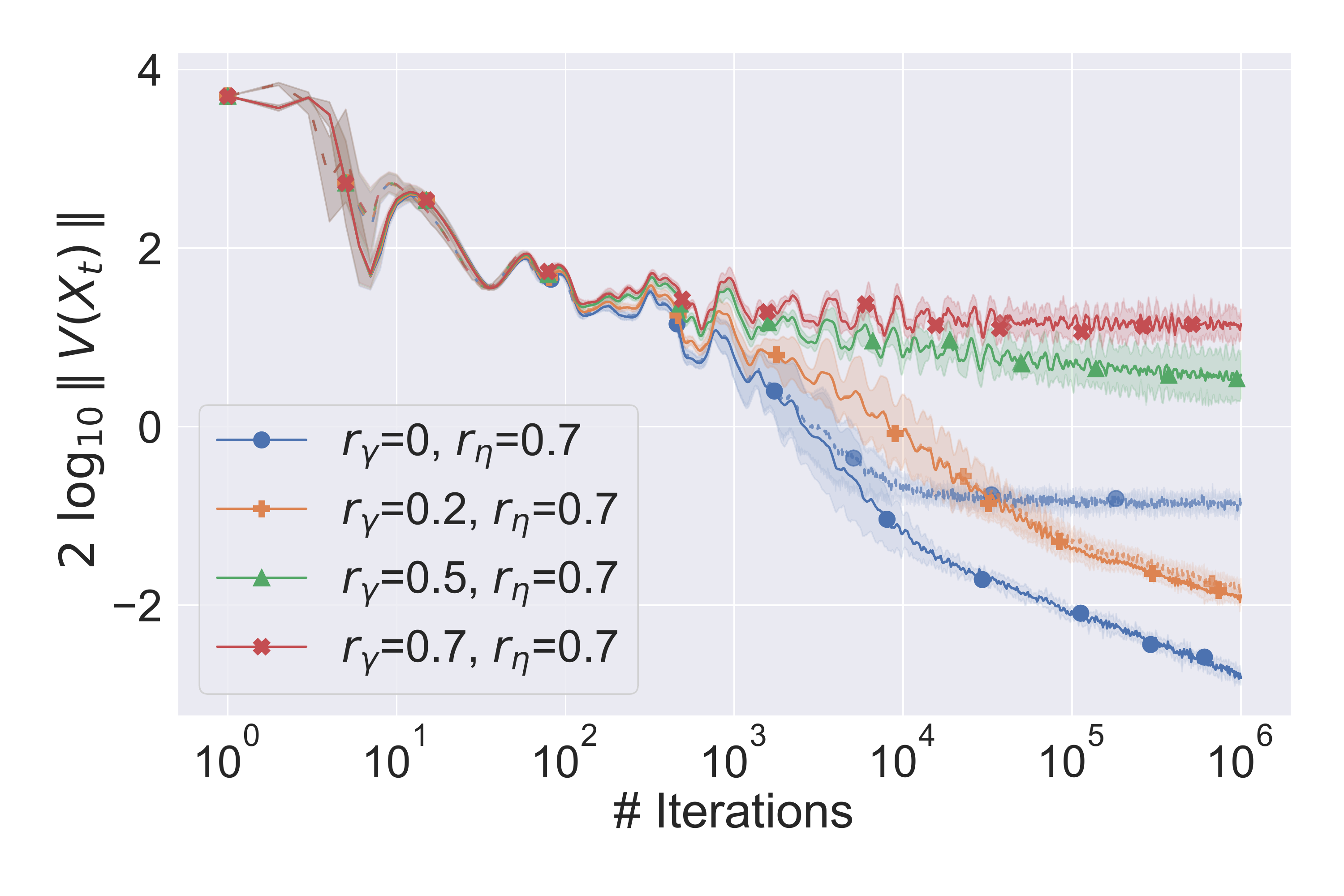}
    \end{subfigure}
    \caption{Convergence of \eqref{eq:DSEG} (top) and \eqref{eq:OG} (bottom) schemes in stochastic bilinear (left), strongly convex-concave (middle) and non convex-concave covariance matrix learning (right) problems.
    In the second row the dashed lines and the solid lines depict respectively the results for optimistic iterates and residual iterates.
    We observe clearly the benefit of \emph{(i)} aggressive exploration and \emph{(ii)} using residual iterates in generalized \ac{OG} methods.
    All curves are averaged over $10$ runs with the shaded area indicating the standard deviation.}
    \label{fig:illustration-full}
\end{figure*}

We provide here a detailed explanation of the problems that we consider in our experiments and elucidate the used parameters.
Additional experimental results are also presented.

\para{Bilinear zero-sum games}
The bilinear zero-sum game takes the form
\begin{equation}
    \sadobj(\minvar,\maxvar) = \minvar^{\top}\!\matlin\maxvar \notag
\end{equation}
where $\matlin$ is a $50\times 50$ invertible matrix in our experiment; in that case, $(\sol[\minvar], \sol[\maxvar]) = (0,0)$ is the only equilibrium point.
We simulate the stochastic oracle by adding a Gaussian noise $\noise\sim\mathcal{N}(0, \noisedev\idm)$ with $\noisedev=0.5$ to the vector field. 



\para{Strongly convex-concave game}
To understand the effect of aggressive exploration in strongly convex-concave problems, we inspect the following example
%
\makeatletter%
\if@twocolumn%
    \begin{align}
    \sadobj(\minvar,\maxvar) &=
    \big(\minvar^\top\!\matmin_2\minvar\big)^2
    + 2\minvar^\top\!\matmin_1\minvar \notag\\
    &~~+ 4 \minvar^\top\!\matlin\maxvar
    - 2\maxvar^\top\!\matmax_1\maxvar
    - \big(\maxvar^\top\!\matmax_2\maxvar\big)^2, \notag
    \end{align}
\else
    \begin{equation}
    \sadobj(\minvar,\maxvar) =
    \big(\minvar^\top\!\matmin_2\minvar\big)^2
    + 2\minvar^\top\!\matmin_1\minvar
    + 4 \minvar^\top\!\matlin\maxvar
    - 2\maxvar^\top\!\matmax_1\maxvar
    - \big(\maxvar^\top\!\matmax_2\maxvar\big)^2, \notag
    \end{equation}
\fi
where $\matmin_1$, $\matmin_2$, $\matmax_1$, $\matmax_2$ are $50\times 50$ positive definite matrices
so $(\sol[\minvar], \sol[\maxvar]) = (0,0)$ is again the only solution of the problem.
We take the same noise distribution to construct the stochastic oracle.


\para{Linear Quadratic Gaussian GAN}
Finally, to examine the convergence of \eqref{eq:DSEG} in stochastic non convex-concave problems, we consider the following problem from \citet{DISZ18} and 
\citet{NK17}:
\begin{align}
    \sadobj(\generatorMat, \discriminatorMat)
    = \ex_{\sampleexp\sim\gaussian(0,\covariance)}[\sampleexp^{\top}\!\discriminatorMat\sampleexp]
    -\ex_{\noiseexp\sim\gaussian(0,\idm)}[
    \noiseexp^{\top}\!\generatorMat^{\top}\!\discriminatorMat\generatorMat\noiseexp]. \notag
\end{align}
This saddle-point problem corresponds to the WGAN formulation without clipping when data are sampled from a normal distribution with covariance matrix $\covariance$, \ie $\sampleexp\sim\gaussian(0,\covariance)$, and the generator and the discriminator are respectively defined by $\generator(\noiseexp)=\generatorMat\noiseexp$, $\discriminator(\sampleexp)=\sampleexp^{\top}\!\discriminatorMat\sampleexp$.
The stochasticity is induced by the sampling of $\sampleexp$ and $\noiseexp$.
For the experiments we take a mini-batch of size $128$ and $\sampleexp$ and $\noiseexp$ of dimension $10$.
As the game may possess multiple equilibria, the squared norm of $\vecfield$ is traced as the convergence measure.

\begin{table}[t]
    \centering
    \small
    \renewcommand{\arraystretch}{1.1}
    \begin{tabularx}{\textwidth}{l*{6}{>{\centering\arraybackslash}X}}
    \toprule
    & \multicolumn{3}{c}{Double stepsize extragradient \eqref{eq:DSEG}}
    & \multicolumn{3}{c}{Generalized optimistic gradient \eqref{eq:OG}}\\
    \cmidrule(lr){2-4}
	\cmidrule(lr){5-7}
	& $\step_{\start}$ & $\stepalt_{\start}$ & $\stepoffset$
	& $\step_{\start}$ & $\stepalt_{\start}$ & $\stepoffset$\\
	\midrule
    Bilinear & 1 & 0.1 & 19 & 0.5 & 0.05 & 19 \\
    Strongly convex-concave & 0.1 & 0.05 & 19 & 0.1 & 0.05 & 19 \\
    Gaussian GAN & 0.5 & 0.05 & 49 & 0.05 & 0.025 & 99\\
    \bottomrule
    \end{tabularx}
    \vspace{2ex}
    \caption{The stepsize parameters for \eqref{eq:DSEG} and \eqref{eq:OG} in the experiments.}
    \label{tab:parameters}
    \vskip -1.5em
\end{table}

\para{Results for \eqref{eq:DSEG} and \eqref{eq:OG}}

Following the discussion of \cref{app:beyond-eg}, we complement the illustration of our method \eqref{eq:DSEG} by a comparison with \eqref{eq:OG} with properly chosen outputs. 
In the experiments, both \eqref{eq:DSEG} and \eqref{eq:OG} are run with stepsize of the form \eqref{eq:stepsizes} with various $\exponent_{\step}$ and $\exponent_{\stepalt}$.
In order to start with the same value for different exponents, we fix $\stepoffset, \step_1,$ and $\stepalt_1$ as indicated in \cref{tab:parameters}, from which we deduce $\step = \step_1 (1+\stepoffset)^{\exponent_{\step}}$ and $\stepalt = \stepalt_1 (1+\stepoffset)^{\exponent_{\stepalt}}$.

As shown in \cref{fig:illustration-full}, for bilinear game and Gaussian GAN examples, the convergence speed of \eqref{eq:DSEG} is positively related to the difference $\exponent_{\step}-\exponent_{\stepalt}$, as per our analysis. For the strongly convex-concave problem, the vanilla \eqref{eq:EG} already achieves $\bigoh(1/\run)$ convergence, and the plot shows that using aggressive $\seqinf{\step}{\run}$ has little influence on it.

Regarding \eqref{eq:OG} with the residual iterates, the algorithm has roughly the same convergence behavior as for \eqref{eq:DSEG}.
In contrast, the optimistic iterates tend to converge much slower.
In particular, choosing a constant exploration step gives the fastest convergence of the residual iterate though the optimistic iterate does not converge, in line with our discussion in \cref{app:beyond-eg}.

\para{Additional discussions for bilinear games}

Few algorithms provably converge in stochastic bilinear games, and among them there are \ac{SHGD} \cite{LBJV+20} and gradient descent with anchoring \cite{RYY19}.
In \cref{fig:bilinear_compare} we illustrate the  convergences of \ac{DSEG} and these two algorithms for the stochastic bilinear saddle-point example.
For \eqref{eq:DSEG} we adopt the optimal stepsize schedule as described in \cref{thm:dsseg-affine}-\ref{thm:dsseg-affine-b}.
The leading stepsize is set to constant $\current[\step]\equiv1$ and the update stepsize is $\current[\stepalt]=\stepalt/(\run+\stepoffset)$ with $\stepalt=2$ and $\stepoffset=19$.
The same $\seqinf{\stepalt}{\run}$ is also used as the stepsize of \ac{SHGD}, in accordance with the decreasing stepsize strategy presented in \cite{LBJV+20}.
As for the anchored gradient methods, its update is written as
\begin{equation}
\notag
    \update
    = \current - \frac{1-\exponent}{\run^{\exponent}} + \frac{(1-\exponent)\step}{\run^\exponentalt}(\state_{\start}-\current),
\end{equation}
and it is proved to converge in all stochastic monotone problems for $\step>0$ and $\exponent,\exponentalt\in(1/2,1)$.
Since no explicit rate is proven for this algorithm when stochastic gradients are used,
\begin{wrapfigure}{br}{0.4\textwidth}
    \centering
    \includegraphics[width=0.4\textwidth]{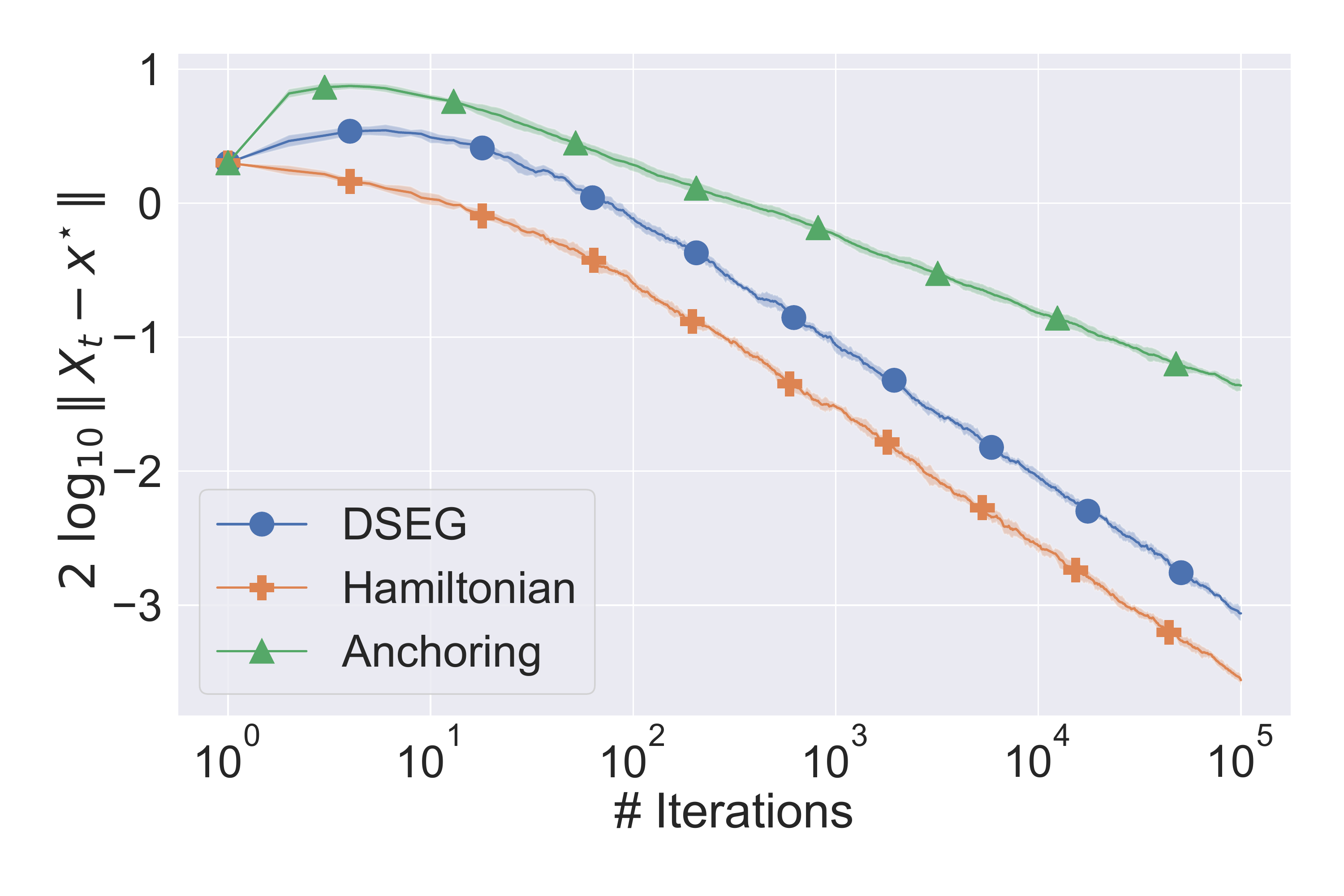}
    \caption{Comparison of \ac{DSEG}, stochastic Hamiltonian gradient descent and anchored gradient in the stochastic bilinear example. All curves are averaged over $10$ runs with the shaded area indicating the standard deviation.}
    \label{fig:bilinear_compare}
    \vskip -1.5em
\end{wrapfigure}
we run hyperparameter optimization to search for the best $\step,\exponent$ and $\exponentalt$, and end up with $\step=1, \exponent=0.7, \exponentalt=0.9$.

\cref{fig:bilinear_compare} confirms that asymptotically both \ac{DSEG} and \ac{SHGD} converge in $\bigoh(1/\run)$ as predicted by the theory.
\ac{SHGD} converges slightly faster than \ac{DSEG} for the first few iterations as it circumvents the rotational dynamics by directly performing stochastic gradient descent on $\norm{\vecfield(\cdot)}^2$, which turns out to be a positive definite quadratic form when $\vecfield$ is linear.
This however comes at the cost of the use of second-order information.
In fact, $\ac{SHGD}$ requires access to an unbiased estimator of $\Jac_{\vecfield}^\top\vecfield$ at every iteration.
Finally, anchoring converges much slower compared to these two methods.
Without further theoretical investigation we do not know if this kind of algorithms can achieve the same $\bigoh(1/\run)$ convergence rate in this problem.

\section{Technical lemmas}
\label{app:tech-lemmas}

In this section we recall several important lemmas that are frequently used in the analysis of stochastic iterative methods.
The first three lemmas on numerical sequences are useful for deriving convergence rates of the algorithms.
See \eg \citet{Pol87} for an abundance of results of this type.

\begin{lemma}
\label{lem:cst-noise-descent}
Let $\seqinf{\seqitem}{\run}$ be a sequence of real numbers
such that for all $\run$,
\begin{equation}
\notag
    \update[\seqitem]
    \le
    (1-\consc)\current[\seqitem] +\conscalt,
\end{equation}
where $1 > \consc > 0$ and $\conscalt > 0$.
Then,
\begin{equation}
\notag
    \seqitem_\run
    \le
    (1-\consc)^{\run-1}\seqitem_1+\frac{\conscalt}{\consc}.
\end{equation}
\end{lemma}

The above lemma comes into play when an algorithm is run with constant stepsize sequences, whereas we resort to the following two lemmas in case of decreasing stepsize sequences of the form \eqref{eq:stepsizes}.

\vspace{0.5em}
\begin{lemma}[{{\citet[Lemma~1]{Chu54}}}]
\label{lem:chung1954}
Let $\seqinf{\seqitem}{\run}$ be a sequence of real numbers
and $\stepoffset\in\N$
such that for all $\run$,
\begin{equation}
\notag
    \label{eq:chung_rec}
    \update[\seqitem]
    \le
    \left(1-\frac{\consc}{\run+\stepoffset}\right)\current[\seqitem]
    + \frac{\conscalt}{(\run+\stepoffset)^{\exponent+1}},
\end{equation}
where $\consc > \exponent > 0$ and $\conscalt > 0$.
Then,
\begin{equation}
\notag
    \label{eq:chung_bound}
    \seqitem_\run
    \le
    \frac{\conscalt}{\consc-\exponent}\frac{1}{\run^\exponent} + \smalloh\left(\frac{1}{\run^\exponent}\right).
\end{equation}
\end{lemma}

\vspace{0.5em}
\begin{lemma}[{{\citet[Lemma~4]{Chu54}}}]
\label{lem:chung1954-bis}
Let $\seqinf{\seqitem}{\run}$ be a sequence of real numbers
and $\stepoffset\in\N$
such that for all $\run$,
\begin{equation}
    \notag
    \update[\seqitem]
    \le
    \left(1-\frac{\consc}{(\run+\stepoffset)^{\exponentalt}}\right)\current[\seqitem]
    + \frac{\conscalt}{(\run+\stepoffset)^{\exponent+\exponentalt}},
\end{equation}
where $1 > \exponentalt > 0$ and $\exponent,\consc,\conscalt > 0$.
Then,
\begin{equation}
\notag
    \seqitem_\run
    =
    \bigoh\left(\frac{1}{\run^\exponent}\right).
\end{equation}
\end{lemma}

To establish almost sure convergence of the iterates, we rely on the Robbins\textendash Siegmund theorem which apply to non-negative almost-supermatingales.

\vspace{0.5em}
\begin{lemma}[\citet{RS71}]
\label{lem:Robbins-Siegmund}
Consider a filtration $\seqinf{\filter}{\run}$ and four non-negative $\seqinf{\filter}{\run}$-adapted
processes $\seqinf{U}{\run}$, $\seqinf{\lambda}{\run}$, $\seqinf{\chi}{\run}$, $\seqinf{\zeta}{\run}$
such that $\sum_\run\current[\lambda] < \infty$ and $\sum_\run\current[\chi] < \infty$ with probability one and
$\forall \run\in\N$,
\begin{equation}
    \label{eq:Robbins-Siegmund}
    \exof{\update[U]\given{\current[\filter]}}
    \le (1+\current[\lambda])\current[U] + \current[\chi] - \current[\zeta].
\end{equation}
Then $\seqinf{U}{\run}$ converges almost surely to a random variable $U_\infty$ and
$\sum_\run\current[\zeta]<\infty$ almost surely.
\end{lemma}

\section{Proofs for global convergence results}
\label{app:DssEG-proofs}

We then start with the proofs of the global results to highlight the effect of double stepsize, before tackling the more challenging local convergence analysis.

\subsection{Proof of \cref{prop:non-conv}: failure of stochastic extragradient}
\label{app:DssEG-non-conv-proof}

\PropNonConv*

\begin{proof}
We write the updates of the algorithm
\begin{equation}
    \left\{
        \begin{array}{l}
             \inter[\minvar] = \current[\minvar]-\current[\step]\current[\maxvar]-\current[\step]\current[\snoise] \\
             \inter[\maxvar] = \current[\maxvar]+\current[\step]\current[\minvar]
        \end{array}
    \right.
    ~~~~~~~~
    \left\{
        \begin{array}{l}
             \update[\minvar] = \current[\minvar]-\current[\step]\current[\maxvar]
                                -\current[\step]^2\current[\minvar]-\current[\step]\inter[\snoise] \\
             \update[\maxvar] = \current[\maxvar] + \current[\step]\current[\minvar]
                                -\current[\step]^2\current[\maxvar]-\current[\step]^2\current[\snoise]
        \end{array}
    \right. \notag
\end{equation}
Therefore
\begin{align}
    \update[\minvar]^2 + \update[\maxvar]^2
    &= (1-\current[\step]^2+\current[\step]^4)(\current[\minvar]^2+\current[\maxvar]^2)
        +\current[\step]^2\inter[\snoise]^2
        +\current[\step]^4\current[\snoise]^2 \notag\\
        &~~-2\current[\step]\inter[\snoise]((1-\current[\step]^2)\current[\minvar]-\current[\step]\current[\maxvar])
        -2\current[\step]^2\current[\snoise]((1-\current[\step]^2)\current[\maxvar]+\current[\step]\current[\minvar]). \notag
\end{align}
Taking expectation leads to
\begin{equation}
    \ex[\update[\minvar]^2 + \update[\maxvar]^2]
    = (1-\current[\step]^2+\current[\step]^4)\ex[\current[\minvar]^2+\current[\maxvar]^2]
        + (\current[\step]^2+\current[\step]^4)\noisevar. \notag
\end{equation}
For sake of simplicity, let us denote $\current[\seqitem]=\ex[\current[\minvar]^2+\current[\maxvar]^2]$.
We consider two scenarios:

\vspace{0.4em}
\emph{Case 1: $\current[\gamma]^2\ge1$.}\quad
We have $1-\current[\step]^2+\current[\step]^4\ge1$ and consequently $\update[\seqitem]\ge\current[\seqitem]$.

\emph{Case 2: $\current[\gamma]^2<1$.}\quad
Notice that 
\begin{equation}
    \update[\seqitem]-\frac{(1+\current[\step]^2)\noisevar}{1-\current[\step]^2}
    = (1-\current[\step]^2+\current[\step]^4)\left(\current[\seqitem]-\frac{(1+\current[\step]^2)\noisevar}{1-\current[\step]^2}\right). \notag
\end{equation}
We then set $\current[\distval]=(1+\current[\step]^2)/(1-\current[\step]^2)$.
Since $1-\current[\step]^2+\current[\step]^4<1$, $\update[\seqitem]$ gets closer to
$\current[\distval]\noisevar$ than $\current[\seqitem]$.
In particular, if $\current[\seqitem]<\current[\distval]\noisevar$,
we have $\current[\seqitem]<\update[\seqitem]<\current[\distval]\noisevar$;
otherwise, $\current[\seqitem]\ge\update[\seqitem]\ge\current[\distval]\noisevar$.
As $\current[\distval]\ge1$, the above implies
$\update[\seqitem] \ge \min(\current[\seqitem], \current[\distval]\noisevar) \ge \min(\current[\seqitem], \noisevar)$.

\vspace{0.4em}
To conclude, in the two cases we have $\update[\seqitem] \ge \min(\current[\seqitem], \noisevar)$,
showing $\liminf_{\run\to\infty} \exof{\current[\minvar]^{2} + \current[\maxvar]^{2}} > 0$.



\paragraph*{A remedy with double stepsize extragradient.}

With different stepsizes, the updates of the algorithm write
\begin{equation}
    \left\{
        \begin{array}{l}
             \inter[\minvar] = \current[\minvar]-\current[\step]\current[\maxvar]-\current[\step]\current[\snoise] \\
             \inter[\maxvar] = \current[\maxvar]+\current[\step]\current[\minvar]
        \end{array}
    \right.
    ~~~~~~~~
    \left\{
        \begin{array}{l}
             \update[\minvar] = \current[\minvar]-\current[\stepalt]\current[\maxvar]
                        -\current[\step]\current[\stepalt]\current[\minvar]-\current[\stepalt]\inter[\snoise] \\
             \update[\maxvar] = \current[\maxvar] + \current[\stepalt]\current[\minvar]
                                -\current[\step]\current[\stepalt]\current[\maxvar]-\current[\step]\current[\stepalt]\current[\snoise]
        \end{array}
    \right. \notag
\end{equation}

This now leads to
\begin{align}
    \ex[\update[\minvar]^2 + \update[\maxvar]^2]
    &= ( (1-\current[\step]\current[\stepalt])^2+\current[\stepalt]^2)\ex[\current[\minvar]^2+\current[\maxvar]^2]
        + (\current[\stepalt]^2+\current[\step]^2\current[\stepalt]^2)\noisevar \notag \\
    &= ( 1-2\current[\step]\current[\stepalt] + \current[\stepalt]^2+\current[\step]^2\current[\stepalt]^2)\ex[\current[\minvar]^2+\current[\maxvar]^2]
        + (\current[\stepalt]^2+\current[\step]^2\current[\stepalt]^2)\noisevar . \notag
\end{align}

Taking $\current[\step] = \frac{1}{\run^{\pexp}}$ and $\current[\stepalt] = \frac{1}{\run^{\qexp}}$, we get
\begin{align}
    \ex[\update[\minvar]^2 + \update[\maxvar]^2]
       &= \left( 1- \frac{2}{\run^{(\pexp+\qexp)}} + \frac{1}{{\run}^{2\qexp}}+  \frac{1}{\run^{2(\pexp+\qexp)}}\right)
       \ex[\current[\minvar]^2+\current[\maxvar]^2]
        + \left( \frac{1}{{\run}^{2\qexp}}+  \frac{1}{\run^{2(\pexp+\qexp)}}\right) \noisevar \notag \\
        &\leq \left( 1- \frac{1.5}{\run^{(\pexp+\qexp)}} \right)\ex[\current[\minvar]^2+\current[\maxvar]^2]
        + \frac{2\noisevar}{{\run}^{2\qexp}} \notag \\
        &= \mathcal{O} \left( \frac{1}{\run^{(\qexp-\pexp)}} \right) \notag
\end{align}
where the inequality comes from 
$1- 2/\run^{(\pexp+\qexp)} + 1/{\run}^{2\qexp}+  1/\run^{2(\pexp+\qexp)} \le 1- 1.5/\run^{(\pexp+\qexp)}$
for large enough $\run$ and the last part is an application of either \cref{lem:chung1954} or \cref{lem:chung1954-bis} with $\consc=1.5>\exponent=\qexp-\pexp>0$ (starting at large enough $\run$). 

Hence, $\ex[\current[\minvar]^2 + \current[\maxvar]^2] \to 0$, i.e. we can find a double stepsize choice, with an aggressive extrapolation step and a conservative update step  ($\pexp<\qexp$) such that $(\current[\minvar],\current[\maxvar])\to (0,0) $ in mean squared error.\\
\end{proof}

\subsection{Proof of \cref{lem:dsseg-descent-stoch}}
\LemmaDescent*

\begin{proof}
Let us denote by $\exoft[\cdot] = \exof{\cdot\given{\current[\filter]}}$ the conditional expectation with respect to the filtration up to time $\run$
and $\interdet = \current-\current[\step]\vecfield(\current)$
the leading state that is generated with deterministic update so that
$\inter = \interdet - \current[\step]\current[\noise]$.
We develop
\begin{align}
    \norm{\update-\sol}^2
        & = \norm{\current-\current[\stepalt]\inter[\svecfield]-\sol}^2 \notag \\
        & = \norm{\current-\sol}^2 - 2\current[\stepalt]\product{\inter[\svecfield]}{\current-\sol}
            + \current[\stepalt]^2\norm{\inter[\svecfield]}^2 \notag \\
        & = \norm{\current-\sol}^2 - 2\current[\stepalt]\product{\inter[\svecfield]}{\interdet-\sol}
            - 2\current[\step]\current[\stepalt]\product{\inter[\svecfield]}{\vecfield(\current)}
            + \current[\stepalt]^2\norm{\inter[\svecfield]}^2.
\label{eq:basic}
\end{align}

We would then like to bound the different terms appearing on the \ac{RHS} of the equality.
With the zero-mean assumption \eqref{eq:mean}, conditioning on $\current[\filter]$ leads to
\begin{align}
\exoft[\product{\inter[\svecfield]}{\interdet-\sol}]
&= \exoft[\product{\vecfield(\inter)}{\interdet-\sol}] \notag \\
&= \exoft[\product{\vecfield(\inter)}{\interdet-\current[\step]\current[\noise]-\sol}]
    + \exoft[\product{\vecfield(\inter)}{\current[\step]\current[\noise]}] \notag \\
&= \exoft[\product{\vecfield(\inter)}{\inter-\sol}]
    + \current[\step]\exoft[\product{\vecfield(\inter)-\vecfield(\interdet)}{\current[\noise]}],
\label{eq:eta-term}
\end{align}
where in the last line we use the fact that $\vecfield(\interdet)$ is $\current[\filter]$-measurable so
\[
    \exoft[\product{\vecfield(\interdet)}{\current[\noise]}]
    =\product{\vecfield(\interdet)}{\exoft[\current[\noise]]}=0.
\]
By Lipschitz continuity of $\vecfield$
\begin{align}
    -\product{\vecfield(\inter)-\vecfield(\interdet)}{\current[\noise]}
    \le \norm{\vecfield(\inter)-\vecfield(\interdet)}\norm{\current[\noise]}
    \le \current[\step]\lips\norm{\current[\noise]}^2.
\label{eq:gamma2-eta-term}
\end{align}
On the other hand, $\exoft[\product{\inter[\svecfield]}{\vecfield(\current)}] = \exoft[\product{\vecfield(\inter)}{\vecfield(\current)}]$
and $\exoft[\norm{\inter[\svecfield]}^2] = \exoft[\norm{\vecfield(\inter)}^2] + \exoft[\norm{\inter[\noise]}^2]$.
By $\current[\stepalt]\le\current[\step]$,
Lipschitz continuity of $\vecfield$ and $\current-\inter=\current[\step]\current[\svecfield]$, we get
\begin{align}
& -2\current[\step]\current[\stepalt]\product{\vecfield(\inter)}{\vecfield(\current)}
+ \current[\stepalt]^2\norm{\vecfield(\inter)}^2 \notag \\
& ~~\le
-2\current[\step]\current[\stepalt]\product{\vecfield(\inter)}{\vecfield(\current)}
+ \current[\step]\current[\stepalt]\norm{\vecfield(\inter)}^2 \notag \\
& ~~=
\current[\step]\current[\stepalt]
(\norm{\vecfield(\current)-\vecfield(\inter)}^2-\norm{\vecfield(\current)}^2) \notag \\
& ~~\le
\current[\step]^3\current[\stepalt]\lips^2\norm{\current[\svecfield]}^2
-\current[\step]\current[\stepalt]\norm{\vecfield(\current)}^2,
\label{eq:gamma-eta-term}
\end{align}
Similar to before we may write
$\exoft[\norm{\current[\svecfield]}^2] = \exoft[\norm{\vecfield(\current)}^2] + \exoft[\norm{\current[\noise]}^2]$.
Therefore, combining \eqref{eq:basic}, \eqref{eq:eta-term}, \eqref{eq:gamma2-eta-term}, \eqref{eq:gamma-eta-term}, we deduce the following
%
\begin{align}
    \exoft[\norm{\update-\sol}^2]
        & \le \norm{\current-\sol}^2
        - 2\current[\stepalt]\exoft[\product{\vecfield(\inter)}{\inter-\sol}]  
        - (\current[\step]\current[\stepalt]-\current[\step]^3\current[\stepalt]\lips^2)
        \norm{\vecfield(\current)}^2\notag\\
        & ~~ + (2\current[\step]^2\current[\stepalt]\lips
            + \current[\step]^3\current[\stepalt]\lips^2)
            \ex[\norm{\current[\noise]}^2]
        + \current[\stepalt]^2 \ex[\norm{\inter[\noise]}^2]. 
    \label{eq:dseg-descent-secondlast}
\end{align}

To finish the proof, we would like to bound the noise terms. Using \eqref{eq:variance} and Jensen's inequality (recall that $\probmoment\ge2$), we have
\begin{equation}
    \label{eq:current-noise-bound}
    \ex[\norm{\current[\noise]}^2] \le (\noisedev+\varcontrol\norm{\current-\sol})^2
    \le 2\noisevar+2\varcontrol^2\norm{\current-\sol}^2.
\end{equation}

Similarly,
\begin{equation}
    \label{eq:inter-noise-bound}
    \begin{aligned}
    \ex[\norm{\inter[\noise]}^2]
    & \le 2\noisevar+2\varcontrol^2\norm{\inter-\sol}^2 \\
    & \le 2\noisevar + 
        4\varcontrol^2\norm{\inter-\current}^2
        +4\varcontrol^2\norm{\current-\sol}^2 \\
    & \le 4\current[\step]^2\varcontrol^2\norm{\current[\svecfield]}^2
        +4\varcontrol^2\norm{\current-\sol}^2 + 2\noisevar\\
    & \le 8\current[\step]^2\varcontrol^2\norm{\vecfield(\current)}^2
        + 16\current[\step]^2\varcontrol^2\noisevar
        + 16\current[\step]^2\varcontrol^4\norm{\current-\sol}^2
        + 4\varcontrol^2\norm{\current-\sol}^2 + 2\noisevar
    \end{aligned}
\end{equation}
Substituting \eqref{eq:inter-noise-bound} and \eqref{eq:current-noise-bound} in \eqref{eq:dseg-descent-secondlast}, we obtain
\begin{align}
    \exoft[\norm{\update-\sol}^2]
        & \le (1 + 4\current[\step]^2\current[\stepalt]\lips\varcontrol
        + 2\current[\step]^3\current[\stepalt]\lips^2\varcontrol
        + 4\current[\stepalt]^2\varcontrol^2
        + 16\current[\step]^2\current[\stepalt]^2\varcontrol^4
        )\norm{\current-\sol}^2 \notag\\
        & ~~
        - 2\current[\stepalt]\exoft[\product{\vecfield(\inter)}{\inter-\sol}]\notag\\
        & ~~
        - (\current[\step]\current[\stepalt]
        -\current[\step]^3\current[\stepalt]\lips^2
        - 8\current[\step]^2\current[\stepalt]^2\varcontrol^2)
        \norm{\vecfield(\current)}^2\notag\\
        & ~~ + (4\current[\step]^2\current[\stepalt]\lips
            + 2\current[\step]^3\current[\stepalt]\lips^2
            + 2\current[\stepalt]^2
            + 16\current[\step]^2\current[\stepalt]^2\varcontrol^2)\noisevar. \notag
\end{align}
We recover \eqref{eq:dsseg-descent-stoch} by using $2\current[\stepalt]^2\noisevar\le4\current[\stepalt]^2\noisevar$.
\end{proof}

\subsection{Proof of \cref{thm:dsseg-general-as}}
\label{app-sub:proof-as}

\ThmGlobalAs*

\begin{proof}
The proof is divided into three key steps.

\vspace{0.3em}
(1) \textit{With probability $1$, $\liminf_{\run\rightarrow\infty} \norm{\vecfield(\current)} = 0$}.\quad
Let $\sol\in\sols$. Using \cref{lem:dsseg-descent-stoch} and \cref{asm:mono}, we get the following
\begin{equation}
\notag
\begin{aligned}
    \exof{\norm{\update - \sol}^{2} \given \current[\filter]}
	&
	\leq (1+\current[\Cons]\varcontrol^2)\norm{\current-\sol}^{2}
	- 2\current[\stepalt] \exof{\product{\vecfield(\inter)}{\inter-\sol} \given \current[\filter]}
	\\
	&
	- \current[\step] \current[\stepalt] (
	    1-\current[\step]^{2}\lips^{2}
	    -8\current[\step]\current[\stepalt]\varcontrol^2)
		\norm{\vecfield(\current)}^{2}
	+ \current[\Cons]\noisevar,
	\\
	&
	\leq (1+\current[\Cons]\varcontrol^2)\norm{\current-\sol}^{2}
	- \current[\step] \current[\stepalt] (
	    1-\current[\step]^{2}\lips^{2}
	    -8\current[\step]\current[\stepalt]\varcontrol^2)
		\norm{\vecfield(\current)}^{2}
	+ \current[\Cons]\noisevar
    \label{eq:dsseg-descent-stoch-app}
\end{aligned}
\end{equation}

Since $\current[\step] < 1/3\max(\smooth, \varcontrol)$ and $\current[\stepalt]\le\current[\step]$, the coefficient
$\current[\consalt]\defeq
\current[\step]\current[\stepalt]
-\current[\step]^3\current[\stepalt]\lips^2
-8\current[\step]^2\current[\stepalt]^2\varcontrol^2$
is non-negative.
Recalling that 
$\current[\Cons]=
    4\current[\step]^2\current[\stepalt]\lips   
    + 2\current[\step]^3\current[\stepalt]\lips^2
    + 4\current[\stepalt]^2
    + 16\current[\step]^2\current[\stepalt]^2\varcontrol^2$,
from our stepsize conditions $\sum_{\run}\current[\stepalt]^2<\infty$, $\sum_{\run}\current[\step]^2\current[\stepalt]<\infty$
and $\seqinf{\step}{\run}$ being upper-bounded, it holds $\sum_{\run}\current[\Cons]<\infty$.
We can therefore apply the Robbins\textendash Siegmund theorem (\cref{lem:Robbins-Siegmund})
to get that (\emph{i}) $\norm{\current-\sol}$ converges
almost surely and (\emph{ii})
$\sum_\run \current[\consalt]\norm{\vecfield(\current)}^2 < \infty$ almost surely.
As the stepsize conditions also imply $\sum_{\run}\current[\consalt]=\infty$, using (\emph{ii}), we deduce immediately 
$\liminf_{\run\rightarrow\infty} \norm{\vecfield(\current)} = 0$ almost surely.

\vspace{0.3em}
(2) \textit{With probability $1$, $\norm{\current-\sol}$ converges for all $\sol\in\sols$}.\quad
In other words, we would like to prove the existence of an event $\singleevent\subset\samples$ satisfying $\prob(\singleevent)=1$ and that for every realization of the event and every $\sol\in\sols$, $\norm{\current-\sol}$ converges.
Since $\vecspace$ is a separable metric space, $\sols$ is also separable and we can find a countable set $\countable$ such that $\sols=\cl(\countable)$ ($\sols$ is closed by continuity of $\vecfield$).
We claim that the choice $\singleevent=\{\norm{\current-\cpoint}\textit{ converges for all }\cpoint\in\countable\}$ is the good candidate.

In effect, taking an arbitrary $\cpoint$ from $\countable$, from \emph{(i)} we know that 
\[\prob(\{\norm{\current-\cpoint}\textit{ converges}\})=1.\]
Therefore from the countability of $\countable$ we have $\prob(\singleevent)=1$.
We now fix $\sol\in\sols$. As $\countable$ is dense in $\sols$, there exists a sequence $\seqinf{\cpoint}{\indg}$ of points in $\countable$ such that $\lim_{\indg\to\infty}\cpoint_\indg=\sol$.
Consider a realization of $\singleevent$, for every $\cpoint_\indg$ we have $\lim_{\run\to\infty}\norm{\current-\cpoint_\indg}=\distval_\indg$ for some $\distval_\indg \ge 0$.
The triangular inequality gives
\begin{equation}
    -\norm{\cpoint_\indg-\sol}
    \le
    \norm{\current-\sol}-\norm{\current-\cpoint_\indg}
    \le
    \norm{\cpoint_\indg-\sol} \notag
\end{equation}
for all $\indg,\run\in\N$. Consequently, for all $\indg\in\N$,
\begin{align}
    -\norm{\cpoint_\indg-\sol}
    &\le
    \liminf_{\run\to\infty}\norm{\current-\sol}
    -\lim_{\run\to\infty}\norm{\current-\cpoint_\indg} \notag \\
    &=
    \liminf_{\run\to\infty}\norm{\current-\sol} - \distval_\indg \notag \\
    &\le
    \limsup_{\run\to\infty}\norm{\current-\sol} - \distval_\indg \notag \\
    &=
    \limsup_{\run\to\infty}\norm{\current-\sol}
    -\lim_{\run\to\infty}\norm{\current-\cpoint_\indg}
    \le \norm{\cpoint_\indg-\sol}. \notag
\end{align}
Taking the limit as $\indg\to\infty$ we obtain the convergence of $\seqinf{\norm{\current-\sol}}{\run}$; more precisely, $\lim_{\run\to\infty}\norm{\current-\sol}=\lim_{\indg\to\infty}\distval_\indg$.
We have thus proved $\singleevent$ satisfies the requirements.

\vspace{0.3em}
(3) \textit{Conclude}.\quad
Combining the points (1) and (2), we get
\[
\prob(\singleevent\intersect\{\liminf_{\run\rightarrow\infty} \norm{\vecfield(\current)} = 0\})=1.
\]
Let us take a realization of this event.
It holds $\liminf_{\run\to\infty}\norm{\vecfield(\current)}=0$ and we can thus extract a subsequence
$(\state_{\extr})_{\run\in\N}$ such that $\lim_{\run\to\infty}\norm{\vecfield(\state_{\extr})}=0$.
Let $\sol\in\sols$, we know that $\norm{\current-\sol}$ converges, implying that $\seqinf{\state}{\run}$ is bounded.
As $\vecspace$ is finite dimensional, we can then further extract
$(\state_{\extrsnd})_{\run\in\N}$ so that $\lim_{\run\to\infty}\state_{\extrsnd} = \point_\infty$ for some $\point_\infty\in\vecspace$.
By continuity of $\vecfield$, we have $\vecfield(\point_\infty)=0$, \ie $\point_\infty\in\sols$.
By the choice of $\singleevent$, we have the convergence of $\seqinf{\norm{\current-\point_\infty}}{\run}$,
and
\begin{equation}
    \lim_{\run\to\infty}\norm{\current-\point_\infty} = \lim_{\run\to\infty}\norm{\state_{\extrsnd}-\point_\infty} = \norm{\point_\infty-\point_\infty} = 0. \notag
\end{equation}
To conclude, we have proved that that $\current$ converges to some $\sol\in\sols$ almost surely.
\end{proof}

\subsection{Proof of \cref{thm:dsseg-general-rate}}
{
\renewcommand{\thetheorem}{\ref{thm:dsseg-general-rate}}%
\begin{theorem}
Suppose that \cref{asm:Lipschitz,asm:noise,asm:mono,asm:err-bound} hold and assume 
that $\current[\step]\le\cons/\lips$ with $\cons<1$.
Then:
\begin{enumerate}
\item
If \eqref{eq:DSEG} is run with $\current[\step] \equiv \step$, $\current[\stepalt] \equiv \stepalt$, we have:
\begin{equation}
\exof{\dist(\current, \sols)^2}
	\leq (1-\minuscst)^{\run-1}\dist(\state_\start, \sols)^2
		+ \frac{\pluscst}{\minuscst} \notag
\end{equation}
with constants $\pluscst=(2\step^2\stepalt\lips+\step^3\stepalt\lips^2+\stepalt^2)\noisevar$ and $\minuscst=\step\stepalt\errbcst^2(1-\cons^2)$.
\item
If \eqref{eq:DSEG} is run with 
$\current[\step] = \step/(\run+\stepoffset)^{1-\exponentalt}$ and $\current[\stepalt] = \stepalt/(\run+\stepoffset)^{\exponentalt}$ for some $\exponentalt\in(1/2,1)$, we have:
\begin{equation}
\exof{\dist(\current, \sols)^{2}}
	\leq
	\frac{\pluscst}{\minuscst-\exponent} \frac{1}{\run^\exponent}
	+ \smalloh\left(\frac{1}{\run^\exponent}\right) \notag
\end{equation}
where $\exponent = \min(1-\exponentalt, 2\exponentalt-1)$ and we further assume that $\step\stepalt\errbcst^2(1-\cons^2) > \exponent$.
In particular, the optimal rate is attained when $\exponentalt=2/3$, which gives $\ex[\dist(\current, \sols)^{2}]=\bigoh(1/\run^{1/3})$.
\end{enumerate}
For the sake of readability, the involved constants are stated for the case $\varcontrol=0$. 
On the other hand, if $\noisedev=0$ and $\varcontrol\geq 0$, a geometric convergence can be proved.
\end{theorem}
}
\begin{proof}
We first consider the case $\varcontrol=0$ so that $\ex[\norm{\current[\noise]}^2]\le\noisevar$ and $\ex[\norm{\inter[\noise]}^2]\le\noisevar$.
Since $\current[\step]\le\cons/\lips$, from \eqref{eq:dseg-descent-secondlast} we deduce
\begin{align}
    \exoft[\norm{\update-\sol}^2]
        & \le
        \norm{\current-\sol}^2
        - \current[\step]\current[\stepalt](1-\cons^2)\norm{\vecfield(\current)}^2
        + (2\current[\step]^2\current[\stepalt]\lips
            + \current[\step]^3\current[\stepalt]\lips^2
            + \current[\stepalt]^2)\noisevar. \notag
\end{align}
By concavity of the minimum operator, we then obtain
\begin{align}
    \exoft[\min_{\sol\in\sols}\norm{\update-\sol}^2]
        & \le \min_{\sol\in\sols}\exoft[\norm{\update-\sol}^2] \notag \\
        & \le
        \min_{\sol\in\sols}\norm{\current-\sol}^2
        - \current[\step]\current[\stepalt](1-\cons^2)\norm{\vecfield(\current)}^2\notag\\
        &~~~~+ (2\current[\step]^2\current[\stepalt]\lips
            + \current[\step]^3\current[\stepalt]\lips^2
            + \current[\stepalt]^2)\noisevar. \notag
\end{align}
In other words,
\begin{align}
    \exoft[\dist(\update, \sols)^2]
        & \le
        \dist(\current, \sols)^2
        - \current[\step]\current[\stepalt](1-\cons^2)\norm{\vecfield(\current)}^2
        + (2\current[\step]^2\current[\stepalt]\lips
            + \current[\step]^3\current[\stepalt]\lips^2
            + \current[\stepalt]^2)\noisevar. \notag
    \label{eq:descent-stoch-general}
\end{align}
Using \cref{asm:err-bound} and the law of total expectation, this gives
\begin{align}
    \ex[\dist(\update, \sols)^2]
    & \le
    (1-\current[\step]\current[\stepalt]\errbcst^2(1-\cons^2))\ex[\dist(\current, \sols)^2]
    + (2\current[\step]^2\current[\stepalt]\lips
    + \current[\step]^3\current[\stepalt]\lips^2
    + \current[\stepalt]^2)\noisevar. \notag
\end{align}
Points 
\ref{thm:dsseg-general-rate-a} and \ref{thm:dsseg-general-rate-b} are obtained respectively by applying \cref{lem:cst-noise-descent} and \cref{lem:chung1954}.

For the case $\varcontrol\neq 0$, the term before $\ex[\dist(\current, \sols)^2]$ is replaced by $1+\current[\Cons]\varcontrol^2-\current[\consalt]\errbcst^2$ where $\current[\consalt]=
\current[\step]\current[\stepalt]
-\current[\step]^3\current[\stepalt]\lips^2
-8\current[\step]^2\current[\stepalt]^2\varcontrol^2$
is defined in the proof of \cref{thm:dsseg-general-as}.
In point \ref{thm:dsseg-general-rate-a}, the term $1+\current[\Cons]\varcontrol^2-\current[\consalt]\errbcst^2$ can be made in $(0, 1)$ for $\step$ and $\stepalt$ properly chosen.
Precisely, we need
\begin{equation}
\notag
    (4\step\lips   
    + 2\step^2\lips^2
    + \frac{4\stepalt}{\step}
    + 16\step\stepalt\varcontrol^2)\varcontrol^2
    + (\step^2\lips^2+8\step\stepalt\varcontrol^2)\errbcst^2 < \errbcst^2.
\end{equation}
To prove point \ref{thm:dsseg-general-rate-b}, notice that the conditions of \cref{lem:chung1954} are still verified when $\step,\stepalt$ and $\stepoffset$ are large enough.
For example, if it holds for all $\run$ 
\begin{equation}
\notag
    (4\current[\step]\lips   
    + 2\current[\step]^2\lips^2
    + \frac{4\current[\stepalt]}{\current[\step]}
    + 16\current[\step]\current[\stepalt]\varcontrol^2)\varcontrol^2
    + (\current[\step]^2\lips^2+8\current[\step]\current[\stepalt]\varcontrol^2)\errbcst^2 \le \errbcst^2/2
\end{equation}
and $\step\stepalt\errbcst^2/2 > \exponent$ then \cref{lem:chung1954} can be applied.
Finally, if $\noisedev=0$, the key inequality becomes
\begin{align}
    \exoft[\norm{\update - \sol}^{2}]
	&\leq (1+\current[\Cons]\varcontrol^2)\norm{\current-\sol}^{2}
	- \current[\consalt]\norm{\vecfield(\current)}^{2}.
	\notag
\end{align}
We therefore obtain geometric convergence for $1+\current[\Cons]\varcontrol^2-\current[\consalt]\errbcst^2\in(0,1)$. 
\end{proof}

\subsection{Proof of \cref{thm:dsseg-affine}}
{
\renewcommand{\thetheorem}{\ref{thm:dsseg-affine}}%
\begin{theorem}
Let $\vecfield$ be an affine operator satisfying \cref{asm:mono}, and suppose that \cref{asm:noise} holds.
Take a constant exploration stepsize  $\current[\step] \equiv \step \le\cons/\lips$ with $\cons<1$ (here $\lips$ is the largest singular value of the associated matrix). Then, the iterates $\seqinf{\state}{\run}$ of \eqref{eq:DSEG} enjoy the following rates:
\begin{enumerate}
\item
If the update stepsize is constant $\current[\stepalt]\equiv\stepalt\le\step$, then:
\begin{align}
    \ex[\dist(\current, \sols)^2]
    \leq
	(1-\minuscst)^{\run-1}\dist(\state_\start, \sols)^2
	+ \frac{\pluscst}{\minuscst} \notag
\end{align}
with  $\pluscst=\stepalt^2(1+\cons^2)\noisevar$ and $\minuscst=\step\stepalt\errbcst^2(1-\cons^2)$.
\item
If the update stepsize is of the form $\current[\stepalt]= \stepalt/(\run+\stepoffset)$
for $\stepalt>1/(\errbcst^2\step(1-\cons^2))$ and $\stepoffset>\stepalt/\step$, then:
\begin{equation}
    \ex[\dist(\current, \sols)^{2}]
    \leq
	\frac{\pluscst}{\minuscst-1} \frac{1}{\run}
	+ \smalloh\left(\frac{1}{\run}\right). \notag
\end{equation}
\end{enumerate}
\end{theorem}
For the sake of readability, the involved constants are stated for the case $\varcontrol=0$.
}
\begin{proof}
To focus on the most important points of the proof, we shall consider the case $\varcontrol=0$, while it is straightforward to derive the same kind of result when $\varcontrol>0$ by following the reasoning of previous proofs.
The crucial step here is then the derivation of a stochastic descent inequality in the form of \eqref{eq:dsseg-descent-stoch}.
This is again based on \eqref{eq:basic}.
Writing $\vecfield(\point) = \mat\point + \vvec$, we can expand
\begin{equation}
    \inter[\svecfield]
    = \mat\current - \current[\step]\mat^2\current -\current[\step]\mat\vvec - \current[\step]\mat\current[\noise]
      + \vvec + \inter[\noise]
    = \vecfield(\interdet) - \current[\step]\mat\current[\noise] + \inter[\noise]. \notag
\end{equation}
We recall that $\interdet = \current-\current[\step]\vecfield(\current)$.
Let $\sol\in\sols$. Together with the zero-mean assumption \eqref{eq:mean}, the above shows that
\begin{align}
    \exoft[\product{\inter[\svecfield]}{\interdet-\sol}] &= \product{\vecfield(\interdet)}{\interdet-\sol}, \notag \\
    \exoft[\product{\inter[\svecfield]}{\vecfield(\current)}] &= \product{\vecfield(\interdet)}{\vecfield(\current)}, \notag \\
    \exoft[\norm{\inter[\svecfield]}^2]
    &= \norm{\vecfield(\interdet)}^2 + \exoft[\norm{\current[\step]\mat\current[\noise]}^2] + \exoft[\norm{\inter[\noise]}^2]. \notag
\end{align}
Similar to \eqref{eq:gamma-eta-term}, we write
\begin{align}
& -2\current[\step]\current[\stepalt]\product{\vecfield(\interdet)}{\vecfield(\current)}
+ \current[\stepalt]^2\norm{\vecfield(\interdet)}^2 \notag \\
& ~~\le
-2\current[\step]\current[\stepalt]\product{\vecfield(\interdet)}{\vecfield(\current)}
+ \current[\step]\current[\stepalt]\norm{\vecfield(\interdet)}^2 \notag \\
& ~~=
\current[\step]\current[\stepalt]
(\norm{\vecfield(\current)-\vecfield(\interdet)}^2-\norm{\vecfield(\current)}^2) \notag \\
& ~~\le \current[\step]\current[\stepalt](\current[\step]^2\lips^2-1)\norm{\vecfield(\current)}^2. \notag
\end{align}
We have $\product{\vecfield(\interdet)}{\interdet-\sol}\ge0$ by \cref{asm:mono}
and
$\exoft[\norm{\current[\step]\mat\current[\noise]}^2] + \exoft[\norm{\inter[\noise]}^2]\le(\current[\step]^2\lips^2+1)\noisevar$
by Lipschitz continuity of $\vecfield$ and the finite variance assumption (\ie \eqref{eq:variance} with $\varcontrol=0$).
Taking expectation with respect to $\current[\filter]$ over \eqref{eq:basic} then leads to
\begin{align}
    \exoft[\norm{\update-\sol}^2]
        & \le
        \norm{\current-\sol}^2
        - \current[\step]\current[\stepalt](1-\current[\step]^2\lips^2)
        \norm{\vecfield(\current)}^2
        + \current[\stepalt]^2(\current[\step]^2\lips^2+1)\noisevar \notag \\
        & =
        \norm{\current-\sol}^2
        - \current[\step]\current[\stepalt](1-\cons^2)
        \norm{\vecfield(\current)}^2
        + \current[\stepalt]^2(1+\cons^2)\noisevar. \notag
\end{align}
Proceeding as in the proof of \cref{thm:dsseg-general-rate}, we get
\begin{align}
    \exoft[\dist(\update, \sols)^2]
        & \le
        \dist(\current, \sols)^2
        - \current[\step]\current[\stepalt](1-\cons^2)\norm{\vecfield(\current)}^2
        + \current[\stepalt]^2(1+\cons^2)\noisevar. \notag
\end{align}
Since $\vecfield$ is affine, it verifies the error bound condition \eqref{eq:err-bound}.
Writing $\step$ in the place of $\current[\step]$ and applying the law of total expectation, we obtain
\begin{align}
    \ex[\dist(\update, \sols)^2]
    & \le
    (1-\step\current[\stepalt]\errbcst^2(1-\cons^2))\ex[\dist(\current, \sols)^2] + \current[\stepalt]^2(1+\cons^2)\noisevar. \notag
\end{align}
We conclude with help of \cref{lem:cst-noise-descent} and \cref{lem:chung1954}.
\end{proof}

\section{Proofs for local convergence results}
\label{app:local}
\subsection{Local assumptions}

For sake of clarity, we recall here the local assumptions that will bu used in the local convergence results.

\AsmLipsLoc*
\smallbreak
\AsmNoiseLoc*
\smallbreak
\AsmMonoLoc*
\smallbreak
\AsmErrLoc*

For \cref{asm:noise-loc} in particular, when the neighborhood $\nhd$ is bounded, the term $\varcontrol \norm{\current-\sol}$ is also bounded and therefore, by choosing a larger $\noisedev$ if needed, \eqref{eq:variance-loc} can be simplified to 
\begin{equation}
\notag
\exof{\dnorm{\current[\noise]}^{\probmoment}
\given \current[\filter]}\one_{\{\current\in\nhd\}} \leq 
\noisedev^{\probmoment}~\text{ for all }~ \sol\in\sols.
\end{equation}
We will consider \eqref{eq:variance-loc} under this form in the sequel.

\subsection{Preparatory lemmas}
The proofs of the local statements are much more demanding.
The principle pillar of our analysis is a stability result formally stated in \cref{subsec:stability}. 
To prepare us for the challenge, we start by introducing the following lemma for bounding a recursive stochastic process.
\begin{lemma}
\label{lem:recursive-stoch}
Consider a filtration $\seqinf{\filter}{\run}$ and four $\seqinf{\filter}{\run}$-adapted
processes $\seqinf{\srv}{\run}$, $\seqinf{\srvm}{\run}$, $\seqinf{\srvp}{\run}$, $\seqinf{\snoise}{\run}$
such that $\seqinf{\srvp}{\run}$ is non-negative and the following recursive inequality is satisfied for all $\run\ge\start$
\[
    \update[\srv]
    \le \current[\srv] - \current[\srvm] + \update[\srvp] + \update[\snoise].
\]
Fixing a constant $\Cons>0$, we define the events $\seqinf{\eventA}{\run}$ by
$\eventA_\start \defeq \{\srv_\start\le\Cons/2\}$ and
$\current[\eventA] \defeq \{\current[\srv]\le\Cons\}\intersect\thinspace\{\current[\srvp]\le\Cons/4\}$
for $\run\ge\afterstart$.
We consider also the decreasing sequence of events $\seqinf{\event}{\run}$ defined by
$\event_\run\defeq \bigcap_{\start\le\runalt\le\run} \eventA_\runalt$.
If the following three assumptions hold true
\begin{enumerate}[label=(\roman*), topsep=3pt, itemsep=1pt]
    \item $\forall\run, \current[\srvm]\one_{\current[\event]}\ge0$,
    \item $\forall\run, \exof{\update[\snoise]\given{\current[\filter]}}\one_{\current[\event]}=0$,
    \item $\sum_{\run=\start}^\infty \ex[(\update[\snoise]^2+\update[\srvp])\one_{\current[\event]}]\le\smallproba\noisebound\prob(\eventA_\start)$,
\end{enumerate}
\vspace{-4pt}
where $\noisebound=\min(\Cons^2/16, \Cons/4)$ and $\delta\in(0,1)$, then
$\prob\left(\bigcap_{\run\ge\start}\current[\eventA]~\vert~\eventA_\start\right) \ge 1-\smallproba.$
\end{lemma}

\begin{proof}
Let us start by introducing the following two $\seqinf{\filter}{\run}$-adapted submartingale sequences 
\begin{equation}
    \sumnoise_\run
    \defeq
    \sum_{\runalt=\afterstart}^{\run} \snoise_\runalt ~~\text{ and }~~
    \sumnoiseall_\run
    \defeq \sumnoise_\run^2 + \sum_{\runalt=\afterstart}^{\run}\srvp_\runalt. \notag
\end{equation}
Subsequently, we define an auxiliary sequence of events 
\[\eventalt_\run 
\defeq \eventA_\start \intersect \thinspace
\{\max_{\afterstart\le\runalt\le\run}\sumnoiseall_\runalt\le\noisebound\}\] which is also decreasing.
With this at hand, we are ready to start our proof.

\vspace{0.3em}
(1) \textit{Inclusion $\eventalt_\run\subset\event_\run$}.\quad
We prove the inclusion by induction. The statement is true when $\run=\start$
as $\eventalt_\start=\event_\start=\eventA_\start$.
For $\run\ge2$, we write
\begin{equation}
\label{eq:stoch-rec-sum}
    \current[\srv]
    \le \srv_\start - \sum_{\runalt=\start}^{\run-1}\srvm_\runalt
    + \sum_{\runalt=\afterstart}^{\run-1}\update[\srvp][\runalt] + \sum_{\runalt=\afterstart}^{\run-1}\update[\snoise][\runalt].
\end{equation}
By induction hypothesis, $\eventalt_{\run-1}\subset\event_{\run-1}$,
and thus for all $\runalt\le\run-1$, we have
$\current[\eventalt]\subset\event_{\run-1}\subset\event_\runalt$.
Combining with \emph{(i)} we deduce that for any realization of $\eventalt_{\run}$,
$\sum_{\runalt=\start}^{\run-1}\srvm_\runalt\ge0$.
On the other hand, by definition of $\eventalt_{\run}$, it holds
$\sumnoiseall_\run\one_{\eventalt_\run}\le\noisebound$.
This implies
\begin{gather}\label{eq:rec-stoch-sum-bound}
    \left(\sum_{\runalt=\afterstart}^{\run-1}\update[\snoise][\runalt]\right)\one_{\eventalt_\run}
    =\sumnoise_{\run}\one_{\eventalt_\run}
    \le\sqrt{\noisebound}
    \le\Cons/4,\\
    \left(\sum_{\runalt=\afterstart}^{\run-1}\update[\srvp][\runalt]\right)\one_{\eventalt_\run}
    \le\noisebound
    \le\Cons/4.
\end{gather}
Finally as $\eventalt_{\run}\subset\eventA_{\start}$ we have $\srv_{\start}\one_{\eventalt_\run}\le\Cons/2$.
Therefore, for any realization of $\eventalt_\run$, using \eqref{eq:stoch-rec-sum}  gives
\begin{equation}
    \srv_\run\le\Cons/2-0+\Cons/4+\Cons/4=\Cons. \notag
\end{equation}
In the meantime \eqref{eq:rec-stoch-sum-bound} ensures as well $\current[\srvp]\one_{\eventalt_\run}\le\Cons/4$
and we have thus proven $\eventalt_\run\subset\eventA_\run$.
Using $\eventalt_\run\subset\eventalt_{\run-1}\subset\event_{\run-1}$,
we conclude $\eventalt_\run\subset\event_\run$.

\vspace{0.3em}
(2) \textit{Recursive bound on $\ex[\current[\sumnoiseall][\run]\one_{\last[\eventalt][\run]}]$}.\quad
Since $\last[\eventalt][\run] \subseteq \lastlast[\eventalt][\run] $, it holds $\last[\eventalt][\run] = \setexclude{\lastlast[\eventalt][\run]}{(\setexclude{\lastlast[\eventalt][\run]}{\last[\eventalt][\run]})}$. We can therefore decompose
\begin{align}
    \ex[\current[\sumnoiseall][\run]\one_{\last[\eventalt][\run]}]
    & =
    \ex[(\current[\sumnoiseall][\run]-\last[\sumnoiseall][\run])\one_{\last[\eventalt][\run]}]
    + \ex[\last[\sumnoiseall][\run]\one_{\last[\eventalt][\run]}]
    \notag\\
    & =
    \ex[(\current[\snoise]^2+2\current[\snoise]\last[\sumnoise]+\current[\srvp])
    \one_{\last[\eventalt][\run]}]
    + \ex[\last[\sumnoiseall][\run]\one_{\lastlast[\eventalt][\run]}]
    - \ex[\last[\sumnoiseall][\run]\one_{\setexclude{\lastlast[\eventalt][\run]}{\last[\eventalt][\run]}}]. \notag
\end{align}
From the law of total expectation, $\eventalt_{\run-1}\subset\event_{\run-1}$ and \emph{(ii)} we have
\begin{equation}
    \ex[\current[\snoise][\run]
        \last[\sumnoise][\run] \one_{\last[\eventalt][\run]}]
    = 
    \ex[
        \exof{\current[\snoise][\run]\given{\last[\filter][\run]}}
        \last[\sumnoise][\run]  \one_{\last[\eventalt][\run]}
    ]
    = 0. \notag
\end{equation}
As $\snoise_\run^2+\srvp_\run$ is non-negative, using again $\eventalt_{\run-1}\subset\event_{\run-1}$, we get
\begin{equation}
    \ex[(\current[\snoise]^2+\current[\srvp])\one_{\last[\eventalt][\run]}]
    \le\ex[(\current[\snoise]^2+\current[\srvp])\one_{\last[\event][\run]}]. \notag
\end{equation}
By definition for any realization in $\setexclude{\lastlast[\eventalt][\run]}{\last[\eventalt][\run]}$, it holds $\last[\sumnoiseall][\run]>\noisebound$ and thus
\begin{equation}
    \ex[
        \last[\sumnoiseall][\run]
        \one_{\setexclude{\lastlast[\eventalt][\run]}{\last[\eventalt][\run]}}
    ]
    \ge
        \noisebound
        \ex[
        \one_{\setexclude{\lastlast[\eventalt][\run]}{\last[\eventalt][\run]}}
    ]
    =
    \noisebound
    \prob(\setexclude{\lastlast[\eventalt][\run]}{\last[\eventalt][\run]}). \notag
\end{equation}
Combining the above we deduce the following recursive bound
\begin{equation}
    \ex[\current[\sumnoiseall][\run]\one_{\last[\eventalt][\run]}]
    \le
    \ex[\last[\sumnoiseall][\run]\one_{\lastlast[\eventalt][\run]}]
    + \ex[(\current[\snoise]^2+\current[\srvp])\one_{\last[\event][\run]}]
    - \noisebound\prob(\setexclude{\lastlast[\eventalt][\run]}{\last[\eventalt][\run]}).
    \label{eq:rec-stoch-rec-bound}
\end{equation}

\vspace{0.3em}
(3) \textit{Conclude.}\quad
Summing \eqref{eq:rec-stoch-rec-bound} from $\run=3$ to $\nRuns$ we obtain
\begin{align}
    \ex[\current[\sumnoiseall][\nRuns]\one_{\last[\eventalt][\nRuns]}]
    &\le
    \ex[\sumnoiseall_\afterstart\one_{\eventalt_\start}]
    + \sum_{\run=3}^{\nRuns}
        \ex[(\current[\snoise]^2+\current[\srvp])\one_{\last[\event][\run]}]
    - \noisebound
        \sum_{\run=3}^{\nRuns}
        \prob(\setexclude{\lastlast[\eventalt][\run]}{\last[\eventalt][\run]}) \notag \\
    &= \sum_{\run=\afterstart}^{\nRuns}
        \ex[(\current[\snoise]^2+\current[\srvp])\one_{\last[\event][\run]}]
    - \noisebound\prob(\setexclude{\eventA_\start}{\last[\eventalt][\nRuns]}),
    \label{eq:rec-stoch-rec-bound-sum}
\end{align}
where in the second line we use $\sumnoiseall_\afterstart=\snoise_\afterstart^2+\srvp_\afterstart$, $\eventalt_\start=\event_\start=\eventA_\start$ and
$\setexclude{\eventalt_\start}{\last[\eventalt][\nRuns]}
=\disjUnion_{3\le\run\le\nRuns}
(\setexclude{\lastlast[\eventalt][\run]}{\last[\eventalt][\run]})$ 
with $\disjUnion$ denoting the disjoint union (true since $(\eventalt_\run)_{\run\ge\start}$ is a decreasing sequence of events).
By repeating the same arguments that are used before and using the fact that $\sumnoiseall_\nRuns$ is non-negative, 
\begin{align}
\prob(\setexclude{\eventA_\start}{\current[\eventalt][\nRuns]})
&= 
\prob(\setexclude{\last[\eventalt][\nRuns]}{\current[\eventalt][\nRuns]})
+ \prob(\setexclude{\eventA_\start}{\last[\eventalt][\nRuns]}) \notag \\
&\le
\frac{1}{\noisebound}
\ex[\current[\sumnoiseall][\nRuns]
    \one_{\setexclude{\last[\eventalt][\nRuns]}{\current[\eventalt][\nRuns]}}]
    + \prob(\setexclude{\eventA_\start}{\last[\eventalt][\nRuns]}) \notag \\
&\le
\frac{1}{\noisebound}
\ex[\current[\sumnoiseall][\nRuns]\one_{\last[\eventalt][\nRuns]}]
    + \prob(\setexclude{\eventA_\start}{\last[\eventalt][\nRuns]}).
\label{eq:rec-stoch-final-decomp}
\end{align}
\eqref{eq:rec-stoch-final-decomp}, \eqref{eq:rec-stoch-rec-bound-sum} along with \emph{(iii)} lead to
\begin{equation}
\prob(\setexclude{\eventA_\start}{\current[\eventalt][\nRuns]})
\le
\frac{1}{\noisebound}
\sum_{\run=\afterstart}^{\nRuns}
    \ex[(\current[\snoise]^2+\current[\srvp])\one_{\last[\event][\run]}]
\le \smallproba\prob(\eventA_\start). \notag
\end{equation}
Subsequently,
\begin{equation}
    \probof{\eventalt_\nRuns\given{\eventA_\start}}
    = 1 - \frac{\prob(\setexclude{\eventA_\start}{\current[\eventalt][\nRuns]})}{\prob(\eventA_\start)} \ge 1-\smallproba. \notag
\end{equation}
With $\eventalt_\nRuns\subset\event_\nRuns$ this also gives $\probof{\event_\nRuns\given{\eventA_\start}}\ge1-\smallproba$.
We notice that $\bigcap_{\run\ge\start}\current[\event] = \bigcap_{\run\ge\start}\current[\eventA]$. As $(\event_\run)_{\run\ge1}$ is decreasing, by continuity from above we conclude
\begin{equation}
    \prob\left(\bigcap_{\run\ge\start}\current[\eventA]~\vert~\eventA_\start\right) = \lim_{\run\rightarrow\infty} \probof{\event_\run\given{\eventA_\start}}
    \ge 1-\smallproba. \notag
\end{equation}
\end{proof}


To apply \cref{lem:recursive-stoch}, we establish another quasi-descent lemma which holds without taking expectation values.

\begin{lemma}\label{lem:dsseg-descent-stoch-noexp}
For all $\sol\in\sols$, $\run\in\N$, the iterates generated by \eqref{eq:DSEG} satisfy the following inequality
\begin{align}
    \norm{\update-\sol}^2
        & \le \norm{\current-\sol}^2
        - 2\current[\stepalt]\product{\vecfield(\inter)}{\inter-\sol} \notag \\
        &~~- 2\current[\step]\current[\stepalt]\norm{\vecfield(\current)}(\norm{\vecfield(\current)}-\norm{\vecfield(\interdet)-\vecfield(\current)}) \notag \\
        & ~~
        - 2\current[\stepalt]\product{\inter[\noise]}{\current-\sol}
        -2\current[\step]\current[\stepalt]
        \product{\vecfield(\interdet)}{\current[\noise]} \notag \\
        &~~
        + 2\current[\step]\current[\stepalt]
        \norm{\current[\svecfield]}
        \norm{\vecfield(\inter)-\vecfield(\interdet)}
        + \current[\stepalt]^2\norm{\inter[\svecfield]}^2
    \label{eq:dsseg-descent-stoch-noexp-nolips}
\end{align}
If we assume \cref{asm:Lipschitz-loc} for some solution $\sol$ and that $\current$, $\interdet$, $\inter$ all lie in this neighborhood, then
\begin{align}
    \norm{\update-\sol}^2
        & \le \norm{\current-\sol}^2
        - 2\current[\stepalt]\product{\vecfield(\inter)}{\inter-\sol}
        - 2\current[\step]\current[\stepalt](1-\current[\step]\lips)
        \norm{\vecfield(\current)}^2 \notag \\
        & ~~ - 2\current[\stepalt]\product{\inter[\noise]}{\current-\sol}
        -2\current[\step]\current[\stepalt]
        \product{\vecfield(\interdet)}{\current[\noise]}
        + 2\current[\step]^2\current[\stepalt]\lips
        \norm{\current[\noise]}\norm{\current[\svecfield]}
        + \current[\stepalt]^2\norm{\inter[\svecfield]}^2.
    \label{eq:dsseg-descent-stoch-noexp}
\end{align}
\end{lemma}
\begin{proof}
Similar to \eqref{eq:basic}, we develop
\begin{align}
    \norm{\update-\sol}^2
        & = \norm{\current-\sol}^2 - 2\current[\stepalt]\product{\vecfield(\inter)}{\current-\sol}
        - 2\current[\stepalt]\product{\inter[\noise]}{\current-\sol}
        + \current[\stepalt]^2\norm{\inter[\svecfield]}^2. \notag
\end{align}
We further develop the second term on the \ac{RHS} of the equality
\begin{align}
    \product{\vecfield(\inter)}{\current-\sol}
    &= \product{\vecfield(\inter)}{\inter-\sol}
    + \current[\step]\product{\vecfield(\inter)}{\current[\svecfield]} \notag \\
    &= \product{\vecfield(\inter)}{\inter-\sol}
    + \current[\step]
    \product{\vecfield(\inter)-\vecfield(\interdet)}{\current[\svecfield]}
    + \current[\step]
    \product{\vecfield(\interdet)}{\current[\svecfield]}. \notag
\end{align}
To deal with the last term
\begin{align}
    \product{\vecfield(\interdet)}{\current[\svecfield]}
    &= \product{\vecfield(\interdet)}{\vecfield(\current)}
    + \product{\vecfield(\interdet)}{\current[\noise]} \notag \\
    &= \product{\vecfield(\interdet)-\vecfield(\current)}{\vecfield(\current)}
    + \norm{\vecfield(\current)}^2
    + \product{\vecfield(\interdet)}{\current[\noise]}. \notag
\end{align}
By combing all the above, we readily get \eqref{eq:dsseg-descent-stoch-noexp-nolips} with Cauchy's inequality. 
If \cref{asm:Lipschitz-loc} holds on a set that $\current$, $\interdet$, $\inter$ belong to, we can further bound
\begin{align}
    &2\current[\step]\current[\stepalt]
    \norm{\vecfield(\inter)-\vecfield(\interdet)}\norm{\current[\svecfield]}
    \le 2\current[\step]^2\current[\stepalt]\lips
    \norm{\current[\noise]}\norm{\current[\svecfield]}, \notag \\
    &2\current[\step]\current[\stepalt]
    \norm{\vecfield(\interdet)-\vecfield(\current)}\norm{\vecfield(\current)}
    \le 2\current[\step]^2\current[\stepalt]\lips\norm{\vecfield(\current)}^2, \notag
\end{align}
which gives \eqref{eq:dsseg-descent-stoch-noexp}.
\end{proof}

\subsection{A stability result}
\label{subsec:stability}

The following theorem characterizes the stability of the algorithm around a solution.
The subsequent stepsize condition encompasses the stepsizes employed in \cref{thm:dsseg-local-as} and \cref{thm:dsseg-local} as special cases.
We recall that $\interdet=\current-\current[\step]\vecfield(\current)$.

\begin{theorem}
\label{thm:dsseg-local-stability}
Let $\sol$ be an isolated solution of \eqref{eq:zero} such that \cref{asm:Lipschitz-loc,asm:noise-loc,asm:mono-loc} are satisfied on $\ballr{\sol}{\nhdradius}$ for some $\probmoment>2, \nhdradius>0$.
We fix a tolerance level $\smallproba\in(0,1)$.
For every $\consalt\in(0,1)$, there is a neighborhood $\nhd_\consalt$ of $\sol$ and a constant $\seriesbound>0$ such that if \eqref{eq:DSEG} is initialized at $\state_\start\in\nhd_\consalt$ and 
is run with stepsizes satisfying $\sum_{\run}\current[\step]\current[\stepalt]=\infty$, $\sum_{\run}\current[\stepalt]^2<\seriesbound$, $\sum_{\run}\current[\step]^2\current[\stepalt]<\seriesbound$ and $\sum_{\run}\current[\step]^\probmoment<\seriesbound$, then
\begin{equation}
\eventI_{\infty}^{\consalt}
	= \{\inter\in\ballr{\sol}{\nhdradius}, \current,\interdet\in\ballr{\sol}{\consalt\nhdradius}
	\enspace\text{for all $\run=\running$}\} \notag
\end{equation}
occurs with probability at least $1-\smallproba$, \ie $\probof{\eventI_{\infty}^{\consalt}\given{\state_{\start}\in\nhd_{\consalt}}} \ge 1-\smallproba$.
\end{theorem}

\begin{proof}
We would like to apply \cref{lem:recursive-stoch}, but instead of indexing by $\run\in\N$, we index by $\runalt\in\N/2$.
We invoke \eqref{eq:dsseg-descent-stoch-noexp-nolips} from \cref{lem:dsseg-descent-stoch-noexp} and set the random variables accordingly
\begin{align}
    \underbrace{\norm{\update-\sol}^2}_{\update[\srv]}
        & \le
        \underbrace{\norm{\current-\sol}^2}_{\current[\srv]}
        - 
        \underbrace{2\current[\stepalt]\product{\vecfield(\inter)}{\inter-\sol}}_{\inter[\srvm]} \notag \\
        &~~-
        \underbrace{2\current[\step]\current[\stepalt]\norm{\vecfield(\current)}(\norm{\vecfield(\current)}-\norm{\vecfield(\interdet)-\vecfield(\current)})}_{\current[\srvm]} \notag \\
        & ~~
        +\underbrace{(- 2\current[\stepalt]\product{\inter[\noise]}{\current-\sol})}_{\update[\snoise]}
        +\underbrace{(-2\current[\step]\current[\stepalt]
        \product{\vecfield(\interdet)}{\current[\noise]})}_{\inter[\snoise]} \notag \\
        &~~
        + \underbrace{2\current[\step]\current[\stepalt]
        \norm{\current[\svecfield]}
        \norm{\vecfield(\inter)-\vecfield(\interdet)}
        + \current[\stepalt]^2\norm{\inter[\svecfield]}^2}_{\update[\srvp]}
    \label{eq:def-rvs}
\end{align}

We additionally define $\inter[\srvp]\defeq\current[\step]^{\probmoment}\norm{\current[\noise]}^{\probmoment}$ and $\inter[\srv] \defeq \current[\srv] - \current[\srvm] + \inter[\srvp] + \inter[\snoise]$ so that \eqref{eq:def-rvs} implies $\update[\srv] \le \inter[\srv] - \inter[\srvm] + \update[\srvp] + \update[\snoise]$. With the definition of $\inter[\srv]$ the inequality  \eqref{eq:stoch-rec-sum} is indeed checked with all half integers.
We should now verify that the assumptions \emph{(i)}, \emph{(ii)} and \emph{(iii)} in \cref{lem:recursive-stoch}
are satisfied for a $\Cons$ that is properly chosen.
Let $\vbound$ denote the supremum of $\norm{\vecfield(\point)}$ for $\point\in\nhdalt$
where $\nhdalt=\ballr{\sol}{\alt\nhdradius}$ and $\alt\nhdradius\defeq\consalt\nhdradius$.
We then choose $\Cons\defeq\min(\alt\nhdradius^2/9, 4(\alt\nhdradius/3)^{\probmoment})$.
We also set $\seriesbound$ small enough to guarantee $\current[\step]\le\min(\alt\nhdradius/(3\vbound),1/\lips)$.

\vspace{0.2em}
(a.0) \textit{Inclusion} $\current[\event]\subset\{\current, \interdet \in \nhdalt\}$ \textit{and} $\inter[\event]\subset\{\current, \interdet, \inter \in \nhdalt\}$.\quad
Since $\current[\event]\subset\current[\eventA]$, for any realization of $\current[\event]$, we have $\norm{\current-\sol}^2\le\Cons\le\alt\nhdradius^2/9$.
It follows
\begin{align}
    \norm{\interdet-\sol}^2
    \le 2\norm{\current-\sol}^2 + 2\current[\step]^2\norm{\vecfield(\current)}^2
    \le \dfrac{2\alt\nhdradius^2}{9} + 2\step_{\run}^2\vbound^2
    \le \dfrac{4\alt\nhdradius^2}{9}. \notag
\end{align}
We have shown $\current[\event]\subset\{\current, \interdet \in \nhdalt\}$.
On the other hand,
$\inter[\event]\subset\current[\eventA]\intersect\inter[\eventA]
\subset\{\current[\srv]\le\Cons\}\intersect\thinspace\{\inter[\srvp]\le\Cons/4\}$.
Therefore for any realization of $\inter[\event]$,
\begin{equation}
    \current[\step]^{\probmoment}\norm{\current[\noise]}^{\probmoment}
    =\inter[\srvp]\le\dfrac{\Cons}{4}\le(\alt\nhdradius/3)^{\probmoment}. \notag
\end{equation}
Subsequently,
\begin{align}
    \norm{\inter-\sol}^2
    \le 3\norm{\current-\sol}^2 + 3\current[\step]^2\norm{\vecfield(\current)}^2 + 3\current[\step]^2\norm{\current[\noise]}^2
    \le \frac{\alt\nhdradius^2}{3} + \frac{\alt\nhdradius^2}{3} + 3\left(\frac{\alt\nhdradius}{3}\right)^2
    \le \alt\nhdradius^2. \notag
\end{align}
This proves $\inter[\event]\subset\{\current, \interdet, \inter \in \nhdalt\}$.

\vspace{0.2em}
(a.1) \textit{Assumption (i)}.\quad
We start by $\inter[\srvm]\one_{\inter[\event]}\ge0$.
This is true because $\inter[\event]\subset\{\inter\in\nhdalt\}$ and $\nhdalt\subset\ballr{\sol}{\nhdradius}$, which allows us
to apply \cref{asm:mono-loc} to obtain $\product{\vecfield(\inter)}{\inter-\sol}\ge0$
whenever $\inter[\event]$ occurs.
Similarly, by $\current[\event]\subset\{\current, \interdet \in \nhdalt\}$ and \cref{asm:Lipschitz-loc} we then have 
\begin{equation}
\current[\srvm]\one_{\current[\event]}\ge2\current[\step]\current[\stepalt](1-\current[\step]\lips)\norm{\vecfield(\current)}^2\ge0. \notag
\end{equation}

\vspace{0.2em}
(a.2) \textit{Assumption (ii)}.\quad Immediate from \eqref{eq:mean-loc}, (a.0) and the law of the total expectation.

\vspace{0.2em}
(a.3) \textit{Assumption (iii)}.\quad
By using that $\current[\event]\subset\{\interdet\in\nhdalt\}$ and $\current[\event]\subset\{\current\in\ballr{\sol}{\nhdradius}\}$ , we get
\begin{align}
    \notag
    \ex[\inter[\snoise]^2\one_{\current[\event]}]
    &\le 4\current[\step]^2\current[\stepalt]^2
    \ex[\norm{\vecfield(\interdet)}^2
    \one_{\current[\event]}
    \norm{\current[\noise]}^2
    \one_{\current[\event]}]\\
    &
    \le 4\current[\step]^2\current[\stepalt]^2\vbound^2
    \ex[\norm{\current[\noise]}^2\one_{
        \{\current\in\ballr{\sol}{\nhdradius}\}}]
    \le 4\current[\step]^2\current[\stepalt]^2\vbound^2\noisevar. \notag
\end{align}
For the last inequality we use \eqref{eq:variance-loc} and Jensen's inequality to bound 
$\ex[\norm{\current[\noise]}^2\one_{
        \{\current\in\ballr{\sol}{\nhdradius}\}}]$.
Similarly, 
\begin{align}
\ex[\norm{\current[\noise]}\one_{\{\current\in\ballr{\sol}{\nhdradius}\}}] &\le\noisedev,\notag\\
\ex[\norm{\inter[\noise]}^2\one_{\{\inter\in\ballr{\sol}{\nhdradius}\}}] &\le\noisevar.\notag
\end{align}
%
Using $\inter[\event]\subset\{\current, \interdet, \inter\in\nhdalt\}$ and \cref{asm:Lipschitz-loc} then gives
\begin{align}
    \ex[\update[\srvp]\one_{\inter[\event]}]
    &\le
    2\current[\step]^2\current[\stepalt]\lips
    \ex[\norm{\current[\noise]}
    (\norm{\vecfield(\current)}+\norm{\current[\noise]})\one_{\inter[\event]}] \notag \\
    &~~~~
    + \current[\stepalt]^2\ex[(\norm{\vecfield(\inter)}^2
    +\norm{\inter[\noise]}^2)\one_{\inter[\event]}] \notag
    \\
    &\le
    2\current[\step]^2\current[\stepalt]\lips(
    \ex[\norm{\current[\noise]}^2\one_{\{\current\in\ballr{\sol}{\nhdradius}\}}]
    + \ex[\norm{\current[\noise]}\one_{\{\current\in\ballr{\sol}{\nhdradius}\}}
        \norm{\vecfield(\current)}\one_{\{\current\in\nhdalt\}}]) \notag \\
    &~~~~+ \current[\stepalt]^2(\ex[\norm{\vecfield(\inter)}^2\one_{\{\inter\in\nhdalt\}}] +
    \ex[\norm{\inter[\noise]}^2\one_{\{\inter\in\ballr{\sol}{\nhdradius}\}}]) \notag
    \\
    &\le
     2\current[\step]^2\current[\stepalt]\lips(\vbound\noisedev+\noisevar)
      + \current[\stepalt]^2(\vbound^2+\noisevar).
    \label{eq:local-stay-product-bound}
\end{align}
By similar arguments and in particular by invoking $\inter[\event]\subset\{\current[\srv]\le\Cons\}$ and the definition of $\Cons$, it follows
\begin{align}
    \ex[\update[\snoise]^2\one_{\inter[\event]}]
    &\le \frac{4}{9}\current[\stepalt]^2\alt\nhdradius^2\noisevar, \notag
\end{align}
Combining the above with $\ex[\inter[\srvp]\one_{\current[\event]}]\le\current[\step]^{\probmoment}\noisedev^{\probmoment}$, we have
\begin{align}
    &\sum_{\runalt\in1,3/2,\dotsc}(
        \inter[\snoise][\runalt]^2+\inter[\srvp][\runalt])
        \one_{\current[\event][\runalt]} \notag \\
    &~~~~~~~~~~~~~\le
    \sum_{\run=1}^{\infty}
    \left(
    \current[\step]^{\probmoment}\noisedev^{\probmoment}
    + 2\current[\step]^2\current[\stepalt]\lips(\vbound\noisedev+\noisevar)
    + 4\current[\step]^2\current[\stepalt]^2\vbound^2\noisevar
    + \current[\stepalt]^2(\vbound^2+\noisevar+\frac{4}{9}\alt\nhdradius^2\noisevar)\right) \notag \\
    &~~~~~~~~~~~~~\le
    \left(\noisedev^{\probmoment}
    + 2\lips(\vbound\noisedev+\noisevar)
    + \frac{4}{\lips}\vbound^2\noisevar
    + \vbound^2 + \noisevar + \frac{4}{9}\alt\nhdradius^2\noisevar\right)\seriesbound. \notag
\end{align}
We can thus pick $\seriesbound$ small enough to make \emph{(iii)} verified.

\vspace{0.2em}
(a.4) \textit{Conclude}.\quad
We set $\nhd_{\consalt}=\ballr{\sol}{\sqrt{\Cons/2}}$ so that $\eventA_{\start} = \{\state_{\start}\in\nhd_{\consalt}\}$.
By invoking \cref{lem:recursive-stoch} we get
$\prob\left(\bigcap_{\run\ge\start}\current[\eventA]~\vert~\eventA_{\start}\right) \ge 1-\smallproba$.
Additionally, (a.0) along with $\nhdalt\subset\ballr{\sol}{\nhdradius}$ imply
$\bigcap_{\run\ge\start}\current[\eventA]\subset\eventI_{\infty}^{\consalt}$, concluding the proof.
\end{proof}



\subsection{Proof of \cref{thm:dsseg-local-as}}
\label{subsec:app-local-as}

\ThmLocalAs*

\begin{proof}
Let $\nhdradius>0$, $\consalt\in(0,1)$.
By \cref{thm:dsseg-local-stability}, we know
that if $\eqref{eq:DSEG}$ is run as stated in \cref{thm:dsseg-local-as} with $\exponent_{\step}+\exponent_{\stepalt} \leq 1$, $2\exponent_{\stepalt}>1$, $2\exponent_{\step}+\exponent_{\stepalt}>1$, $\exponent_{\step}\probmoment>1$ and small enough $\step$, $\stepalt$, the event $\eventI_{\infty}^{\consalt}$ occurs with probability $1-\smallproba$.
With this at hand we are ready to prove the large probability convergence result.
For $\run\in\N$, let us define the following events
\begin{align}
\current[\eventI] & \defeq
    \{\current[\state][\runalt],\inter[\tilde{\state}][\runalt]\in\ballr{\sol}{\consalt\nhdradius}\,\text{for all $\runalt=1, 2, \dots, \run$}\} \notag \\
\inter[\eventI] & \defeq
    \current[\eventI] \intersect\thinspace\{\inter[\state][\runalt]\in\ballr{\sol}{\nhdradius} \,\text{for all $\runalt=1, 2, \dots, \run$}\}. \notag
\end{align}
We notice that $\eventI_{\infty}^{\consalt}=\bigcap_{\run\ge1}\inter[\eventI]$.
We would like to establish a recursive inequality in the form of \eqref{eq:Robbins-Siegmund} by taking $\current[U]=\norm{\current-\sol}\one_{\past[\eventI]}$.
The main difficulty consists in controlling the term
$\exoft[\product{\vecfield(\interdet)}{\current[\noise]}\one_{\inter[\eventI]}]$,
which is generally non-zero as $\one_{\inter[\eventI]}$ is not $\current[\filter]$-measurable.
To achieve this, we rely on the following key observation.
\begin{equation}
    \exoft[\current[\noise]\one_{\current[\eventI]}] =
    \exoft[\current[\noise]\one_{\inter[\eventI]}] + \exoft[\current[\noise]\one_{\setexclude{\current[\eventI]}{\inter[\eventI]}}]. \notag
\end{equation}
As $\one_{\current[\eventI]}$ is $\current[\filter]$-measurable and $\current[\eventI]\subset\{\current\in\ballr{\sol}{\nhdradius}\}$, $\exoft[\current[\noise]\one_{\current[\eventI]}]$ is indeed zero and this implies
\begin{equation}
\label{eq:local-rate-ex-eq}
    \norm{\exoft[\current[\noise]\one_{\inter[\eventI]}]} = \norm{\exoft[\current[\noise]\one_{\setexclude{\current[\eventI]}{\inter[\eventI]}}]}.
\end{equation}
The problem then reduces to finding an upper bound of
$\norm{\exoft[\current[\noise]\one_{\setexclude{\current[\eventI]}{\inter[\eventI]}}]}$.
By definition, for any realization of $\setexclude{\current[\eventI]}{\inter[\eventI]}$,
$\interdet\in\ballr{\sol}{\consalt\nhdradius}$ and $\inter\notin\ballr{\sol}{\nhdradius}$.
Since $\inter=\interdet-\current[\step]\current[\noise]$, we deduce
\begin{equation}
    \setexclude{\current[\eventI]}{\inter[\eventI]} \subset \{\norm{\current[\step]\current[\noise]}\ge(1-\consalt)\nhdradius\}. \notag
\end{equation}
Therefore, using $\current[\eventI]\subset\{\current\in\ballr{\sol}{\nhdradius}\}$ along with the Chebyshev's inequality yields
\begin{equation}
    \probof{\setexclude{\current[\eventI]}{\inter[\eventI]}\given{\current[\filter]}}
    \le
    \prob\left(
    \norm{\current[\noise]}\one_{\{\current\in\ballr{\sol}{\nhdradius}\}}
    \ge\frac{(1-\consalt)\nhdradius}{\current[\step]}\mid\current[\filter]
    \right)
    \le \frac{\noisevar\current[\step]^2}{(1-\consalt)^2\nhdradius^2}. \notag
\end{equation}
Applying the Cauchy\textendash Schwarz inequality leads to
\begin{equation}
\label{eq:local-rate-ex-upbound}
    \norm{\exoft[\current[\noise]\one_{\setexclude{\current[\eventI]}{\inter[\eventI]}}]}
    \le
    \sqrt{\exoft[\norm{\current[\noise]\one_{\current[\eventI]}}^2]}
    \sqrt{\exoft[\one_{\setexclude{\current[\eventI]}{\inter[\eventI]}}^2]}
    \le \frac{\noisevar\current[\step]}{(1-\consalt)\nhdradius}.
\end{equation}
Then, by using \eqref{eq:local-rate-ex-eq}, \eqref{eq:local-rate-ex-upbound} and $\inter[\eventI]\subset\current[\eventI]$,
\begin{align}
\exoft[\product{\vecfield(\interdet)}{\current[\noise]}\one_{\inter[\eventI]}]
&= \exoft[\product{\vecfield(\interdet)\one_{\current[\eventI]}}{\current[\noise]\one_{\inter[\eventI]}}] \notag \\
&= \product{\vecfield(\interdet)\one_{\current[\eventI]}}{\exoft[\current[\noise]\one_{\inter[\eventI]}]} \notag \\
&\le \norm{\vecfield(\interdet)\one_{\current[\eventI]}} \norm{\exoft[\current[\noise]\one_{\inter[\eventI]}]} \notag \\
&\le \frac{\vbound\noisevar\current[\step]}{(1-\consalt)\nhdradius},
\label{eq:local-rate-complex-bound}
\end{align}
where $\vbound\defeq\sup_{\point\in\ballr{\sol}{\nhdradius}}\norm{\vecfield(\point)}$.
We can now derive a recursive bound on $\ex[\norm{\update-\sol}\one_{\inter[\eventI]}]$ by invoking \cref{lem:dsseg-descent-stoch-noexp}. 
The inequality $\eqref{eq:dsseg-descent-stoch-noexp}$ multiplied by $\one_{\inter[\eventI]}$ holds true by definition of $\inter[\eventI]$ and \cref{asm:Lipschitz-loc}.
The desired inequality can then be obtained by taking expectation conditioned on $\current[\filter]$.
On the one hand, we use
\begin{gather}
    \product{\vecfield(\inter)}{\inter-\sol}\one_{\inter[\eventI]}\ge0 \notag \\
    \exoft[{\product{\inter[\noise]}{\current-\sol}}\one_{\inter[\eventI]}]
    = \exoft[\product{\ex_{\run+\frac{1}{2}}[\inter[\noise]]\one_{\inter[\eventI]}}{\current-\sol}]=0. \notag
\end{gather}

On the other hand, the last two terms of \eqref{eq:dsseg-descent-stoch-noexp} can be bounded similarly as in
\eqref{eq:local-stay-product-bound} and the antepenultimate term can now be bounded thanks to \eqref{eq:local-rate-complex-bound}.
We then obtain
\begin{align}
    \exoft[\norm{\update-\sol}^2\one_{\inter[\eventI]}]
        & \le \exoft[\norm{\current-\sol}^2\one_{\inter[\eventI]}]
        - 0
        - 2\current[\step]\current[\stepalt](1-\current[\step]\lips)
        \exoft[\norm{\vecfield(\current)}^2\one_{\inter[\eventI]}] \notag \\
        & ~~ - 0
        + 2\current[\step]^2\current[\stepalt]
        \frac{\vbound\noisevar}{(1-\consalt)\nhdradius}
        + \current[\stepalt]^2(\vbound^2+\noisevar)
        + 2\current[\step]^2\current[\stepalt]\lips (\vbound\noisedev + \noisevar).
    \label{eq:stoch-local-condition-rec}
\end{align}
Without loss of generality we may suppose $\current[\step]\lips\le1/2$.
To simplify the notation, we set
\begin{align}
    \current[\srvm] =
    \min\left(\norm{\current-\sol}^2, 
    \current[\step]\current[\stepalt]\norm{\vecfield(\current)}^2\right),
    \enspace
    \boundalt_1
    =
    2\frac{\vbound\noisevar}{(1-\consalt)\nhdradius}+2\lips(\vbound\noisedev + \noisevar),
    \enspace
    \boundalt_2
    =
    \vbound^2+\noisevar. \notag
\end{align}
It follows from \eqref{eq:stoch-local-condition-rec}
\begin{align}
    \exoft[\norm{\update-\sol}^2\one_{\inter[\eventI]}]
        \le \exoft[(\norm{\current-\sol}^2
        - \current[\srvm])
        \one_{\inter[\eventI]}]
        + \current[\step]^2\current[\stepalt]
        \boundalt_1
        + \current[\stepalt]^2\boundalt_2. \notag
\end{align}
As $\norm{\current-\sol}^2 - \current[\srvm]\ge0$ and $\inter[\eventI]\subset\past[\eventI]$, this implies
\begin{align}
    \exoft[\norm{\update-\sol}^2\one_{\inter[\eventI]}]
        \le \norm{\current-\sol}^2\one_{\past[\eventI]}
        - \current[\srvm]\one_{\past[\eventI]}
        + \current[\step]^2\current[\stepalt]
        \boundalt_1
        + \current[\stepalt]^2\boundalt_2. \notag
\end{align}
Invoking the Robbins\textendash Siegmund theorem (\cref{lem:Robbins-Siegmund}) gives the almost sure convergence of $\sum_\run\current[\srvm]\one_{\past[\eventI]}$ and $\norm{\current-\sol}^2\one_{\past[\eventI]}$.
We use $\prob(\eventI_{\infty}^{\consalt})>1-\smallproba$ and deduce that
\[
\prob\left(\underbrace{
\eventI_{\infty}^{\consalt}
\intersect\left\{\sum_{\run=1}^{\infty}\current[\srvm]\one_{\past[\eventI]}<\infty\right\}
\intersect\left\{\norm{\current-\sol}^2\one_{\past[\eventI]} converges\right\}}_{\singleevent}
\right)\ge1-\smallproba. \notag
\]
Since $\eventI_{\infty}^{\consalt}=\bigcap_{\run\ge1}\inter[\eventI]$, for any realization of the above event it holds $\sum_\run\current[\srvm]<\infty$ and $\norm{\current-\sol}^2$ converges.
We assume by contradiction that $\norm{\current-\sol}^2$ converges to some constant $\distval>0$.
From the summability of $\seqinf{\srvm}{\run}$ we know that $\current[\srvm]\to0$ and therefore for all $\run$ large enough we have in fact $\current[\srvm] = \current[\step]\current[\stepalt]\norm{\vecfield(\current)}^2$.
It follows that $\sum_\run\current[\step]\current[\stepalt]\norm{\vecfield(\current)}^2 < \infty$.
Repeating the arguments of \cref{app-sub:proof-as} (Proof of \cref{thm:dsseg-general-as}) we then show that $\norm{\current-\sol}\to0$, which is a contradiction (we take $\nhdradius$ small enough so that $\sol$ is the only solution of \eqref{eq:zero} in $\ballr{\sol}{\nhdradius}$).
We have therefore proved that $\norm{\current-\sol}\to0$ for any realization of $\singleevent$.
In conclusion, $\current$ converges to $\sol$ with probability at least $1-\smallproba$.
\end{proof}


\subsection{Proof of \cref{prop:local-err-bound}}

{
\renewcommand{\theproposition}{\ref{prop:local-err-bound}}%
\begin{proposition}
\label{prop:local-err-bound-2}
If a solution $\sol$ satisfies \cref{asm:err-bound-loc}, then for every $\smallcst>0$, there is a neighborhood $\nhd$ of $\sol$ such that the error bound condition \eqref{eq:err-bound} is satisfied on $\nhd$ with constant $\errbcst=\sing_{\min}-\smallcst$ where $\sing_{\min}$ denotes the smallest singular value of $\Jacf{\vecfield}{\sol}$.
\end{proposition}
}
\begin{proof}
By definition of Jacobian we have
\begin{equation}
    \label{eq:taylor}
    \vecfield(\point) = \vecfield(\sol) + \Jacf{\vecfield}{\sol}(\point-\sol) + \smalloh(\norm{\point-\sol}).
\end{equation}
By the min-max principle of singular value it holds
\begin{equation}
    \label{eq:singular-value}
    \norm{\Jacf{\vecfield}{\sol}(\point-\sol)} \ge \sing_{\min}\norm{\point-\sol}.
\end{equation}
Since $\vecfield(\sol)=0$, combining \eqref{eq:taylor} and \eqref{eq:singular-value} gives
\begin{equation}
    \norm{\vecfield(\point)} \ge \sing_{\min}\norm{\point-\sol} - \smalloh(\norm{\point-\sol}). \notag
\end{equation}
We conclude by noticing $\dist(\point,\sols)=\norm{\point-\sol}$ when $\nhd$ is small enough.
\end{proof}


\subsection{Proof of \cref{thm:dsseg-local}}


\ThmLocalRate*

\begin{proof}
Both \ref{dsseg-local-a} and \ref{dsseg-local-b}
are direct consequences of \cref{thm:dsseg-local-stability}.
In effect, since $\probmoment>3$, the sum of the series $\sum_{\run}\current[\stepalt]^2$, $\sum_{\run}\current[\step]^2\current[\stepalt]$ and $\sum_{\run}\current[\step]^\probmoment$ can be made arbitrarily small by taking sufficiently large $\stepoffset$.
Moreover, $\sol$ is an isolated solution because $\Jacf{\vecfield}{\sol}$ is non-singular.
Therefore, taking $\eventI_{\nhd}\defeq\eventI_{\infty}^{\consalt}$, $\nhd\defeq\nhd^{\consalt}$ and $\nhdalt\defeq\ballr{\sol}{\consalt\nhdradius}$ readily gives \ref{dsseg-local-a} and \ref{dsseg-local-b}.

Finally, to guarantee \ref{dsseg-local-c}, we need to have $\consalt$ small enough and enforce $\step\stepalt\sing_{\min}^2(1-\step_{\start}\lips)>1/6$.
In fact, from $\step\stepalt\sing_{\min}^2(1-\step_{\start}\lips)>1/6$ we deduce the existence of $\smallcst\in(0,\sing_{min})$ such that $\step\stepalt(\sing_{\min}-\smallcst)^2(1-\step_{\start}\lips)>1/6$.
Since $\Jacf{\vecfield}{\sol}$ is non-singular, by \cref{prop:local-err-bound-2} we can choose $\consalt>0$ so that the error bound condition \eqref{eq:err-bound} is satisfied on $\ballr{\sol}{\consalt\nhdradius}$ with $\errbcst=\sing_{\min}-\smallcst$.
Let $\boundalt_1$, $\boundalt_2$ be defined as in \cref{subsec:app-local-as}.
We then obtained from \eqref{eq:stoch-local-condition-rec}
\begin{align}
        \ex[\norm{\update-\sol}^2\one_{\inter[\eventI]}]
        \le (1-2\current[\step]\current[\stepalt]\errbcst^2(1-\current[\step]\lips))
        \ex[\norm{\current-\sol}^2\one_{\inter[\eventI]}]
        + \current[\step]^2\current[\stepalt] \mathcal{M}_1
        + \current[\stepalt]^2 \mathcal{M}_2. \notag
\end{align}
By using $\inter[\eventI]\subset\past[\eventI]$, we get
\begin{align}
        \ex[\norm{\update-\sol}^2\one_{\inter[\eventI]}]
        \le (1-2\current[\step]\current[\stepalt]\errbcst^2(1-\current[\step]\lips))
        \ex[\norm{\current-\sol}^2\one_{\past[\eventI]}]
        + \current[\step]^2\current[\stepalt] \mathcal{M}_1
        + \current[\stepalt]^2 \mathcal{M}_2 \notag
\end{align}
Therefore, with the specified stepsize policy and the condition $\step\stepalt\errbcst^2(1-\step_{\start}\lips)>1/6$,
applying \cref{lem:chung1954} yields $\ex[\norm{\update-\sol}^2\one_{\inter[\eventI]}] = \bigoh(1/\run^{1/3})$. Finally
\begin{align}
    \exof{\norm{\state_{\nRunsNew} - \sol}^{2} \given \eventI_{\infty}^{\consalt}}
    = \frac{\ex[\norm{\state_{\nRunsNew} - \sol}^{2}\one_{\eventI_{\infty}^{\consalt}}]}{\prob(\eventI_{\infty}^{\consalt})}
    \le \frac{\ex[\norm{\state_{\nRunsNew} - \sol}^{2}{\one_{\past[\eventI]}]}}{1-\smallproba}, \notag
\end{align}
which proves $\exof{\norm{\state_{\nRunsNew} - \sol}^{2} \given  \eventI_{\infty}^{\consalt}}=\bigoh(1/\run^{1/3})$.
\end{proof}

\end{document}